\documentclass[leqno]{amsart}

\usepackage{amsmath}
\usepackage{amssymb}
\newcommand{\F}{{\mathbb{F}}}
\newcommand{\C}{{\mathbb{C}}}
\newcommand{\Q}{{\mathbb{Q}}}
\newcommand{\Z}{{\mathbb{Z}}}
\newcommand{\ba}{{\boldsymbol{a}}}
\newcommand{\bA}{{\boldsymbol{A}}}
\newcommand{\fS}{{\mathfrak{S}}}
\newcommand{\cF}{{\mathcal{F}}}
\newcommand{\cL}{{\mathcal{L}}}
\newcommand{\cM}{{\mathcal{M}}}
\newcommand{\cR}{{\mathcal{R}}}

\newcommand{\cLR}{{\mathcal{LR}}}
\newcommand\Irr{\operatorname{Irr}}
\newcommand\Ind{\operatorname{Ind}}
\newcommand\sgn{\varepsilon}
\renewcommand{\leq}{\leqslant}
\renewcommand{\geq}{\geqslant}
\newtheorem{thm}{Theorem}[section]
\newtheorem{cor}[thm]{Corollary}
\newtheorem{prop}[thm]{Proposition}
\newtheorem{lem}[thm]{Lemma}

\theoremstyle{definition}
\newtheorem{exmp}[thm]{Example}
\newtheorem{defn}[thm]{Definition}

\theoremstyle{remark}
\newtheorem{rem}[thm]{Remark}

\begin{document}

\date{}

%\title{On the Kazhdan--Lusztig order on families in type~$B_n$} 
\title{Ordering Lusztig's families in type~$B_n$} 

\author{Meinolf Geck and Lacrimioara Iancu}

\address{Institute of Mathematics, University of Aberdeen, Aberdeen 
AB24 3UE, UK} 

\email{m.geck@abdn.ac.uk, l.iancu@abdn.ac.uk}

\begin{abstract} 
Let $W$ be a finite Coxeter group and $L$ be a weight function on $W$ in 
the sense of Lusztig. We have recently introduced a pre-order relation 
$\preceq_L$ on the set of irreducible characters of $W$ which extends 
Lusztig's definition of ``families'' and which, conjecturally, corresponds 
to the ordering given by Kazhdan--Lusztig cells. Here, we give an explicit 
description of $\preceq_L$ for $W$ of type $B_n$ and any $L$. (All other 
cases are known from previous work.) This crucially relies on some new 
combinatorial constructions around Lusztig's ``symbols''. Combined with 
previous work, we deduce general compatibility results between $\preceq_L$ 
and Lusztig's $\ba$-function, valid for any $W,L$.
\end{abstract}


\maketitle
\pagestyle{myheadings}
%\markboth{Geck and Iancu}{Kazhdan--Lusztig order in type $B_n$}
\markboth{Geck and Iancu}{Ordering Lusztig's families in type $B_n$}

%\noindent {\it Keywords.} Coxeter groups, Lusztig families, bipartitions
%and symbols.

%\noindent 2000 {\it Mathematics Subject Classification.} Primary 20C08. 

%%%%%%%%%%%%%%%%%%%%%%%%%%%%%%%%%%%%%%%%%%%%%%%%%%%%%%%%%%%%%%%%%%%%%%%%%%%
\section{Introduction} \label{sec0}
Let $W$ be a finite Coxeter group and $\Irr(W)$ be the set of (complex) 
irreducible characters of $W$. Given a weight function $L$ on $W$ in the 
sense of Lusztig, one can attach to each $E\in\Irr(W)$ a numerical
invariant $\ba_E$. These invariants are defined using the associated
generic Iwahori--Hecke algebra; they  are an essential ingredient in 
Lusztig's definition \cite{LuBook} of the ``families'' of $\Irr(W)$. In 
\cite{klord}, we have introduced a natural pre-order relation $\preceq_L$ 
on $\Irr(W)$ which extends the definition of families. Furthermore, 
we have shown that if we are in the ``equal parameter case'' (that is, 
$L$ is constant on the generators of $W$) and $W$ is the Weyl group of 
a connected reductive algebraic group $G$, then 
\begin{itemize}
\item $\preceq_L$ corresponds to the ordering given by Kazhdan--Lusztig 
cells; 
\item $\preceq_L$ admits a geometric interpretation in terms of the
Springer correspondence and the closure relation among special
unipotent classes of $G$.
\end{itemize}
In particular, since the latter is known, this yields an explicit 
description of $\preceq_L$ in the equal parameter case. The pre-orders 
associated with arbitrary $L$ are not known to admit such a geometric 
interpretation, but in any case they are highly relevant in the study 
of cellular structures on Hecke algebras, following \cite{mycell}, 
\cite{myedin}. 

Some further examples of $\preceq_L$ have been discussed in \cite{klord} 
for cases where computations are possible (including Coxeter groups of 
type $F_4$ and of non-crystallographic type), but the essential case that 
remains to be considered -- and this is what we will do in this paper -- is 
when $W$ is of type $B_n$ and $L$ is an arbitrary weight function on $W$. 
In this case, the set $\Irr(W)$ is parametrised by pairs of partitions 
$(\lambda, \mu)$ such that $|\lambda|+|\mu|=n$ and there are infinitely 
many weight functions, indexed by two parameters $a,b\in\Z$. 
Thus, $\preceq_L$ gives rise to pre-order 
relations on pairs of partitions, depending on the two parameters $a,b$. 
Our aim is to provide an explicit combinatorial description for these pre-orders.

A familar example of a pre-order relation on pairs of partitions is 
given by the ``dominance order'' which already appeared in the work of 
Dipper--James--Murphy \cite{DJM3}. We shall see that, in our setting, 
this corresponds precisely to the ``asymptotic case'' where  $b$ is
large with respect to $a$ (more precisely, $b>(n-1)a>0$). At another  
extreme where $a=b$, our pre-order admits a geometric 
interpretation as mentioned above. The main results of this paper 
will deal with general choices of $a,b$. See also Bonnaf\'e--Jacon 
\cite{boja} for a discussion of the combinatorics around the 
associated cellular structures in these cases. 

In Section~\ref{sec2}, we recall the precise definition of Lusztig's 
families and the pre-order $\preceq_L$ on $\Irr(W)$. Following 
\cite[\S 6.5]{gepf},  \cite[\S 4]{mycarp}, we work with a definition 
of the invariants $\ba_E$ which is independent of the theory of 
generic Iwahori--Hecke algebras. 

In Section~\ref{sec1}, we study certain pre-order relations on pairs of 
partitions which are defined in a purely combinatorial fashion. This is 
done in the framework provided by the combinatorics of Lusztig's 
``symbols''. The most difficult result about these pre-orders is 
Proposition~\ref{cor11} which generalises a familiar property of the 
dominance order for partitions. As far as we are aware, this is a 
new result; the proof will be given in Sections~\ref{secproof1} and 
\ref{secproof2}. 

In Section~\ref{secasym}, we establish the relation between the
Dipper--James--Murphy dominance order and our pre-order relations.

In Section~\ref{sec3}, we prove the main results for $W$ of type $B_n$,
which show that the combinatorial constructions in Section~\ref{sec1} 
indeed describe the pre-order $\preceq_L$. This will be complemented 
in Section~\ref{sec3a} by the discussion of examples and further 
interpretations of $\preceq_L$ in type $B_n$. We conclude by establishing 
some general properties of $\preceq_L$ in Section~\ref{sec4}.

Combining these new results with the known ones from \cite{klord}, we 
can draw some general conclusions about the pre-order relation 
$\preceq_L$, for any $W$ and any weight function $L$ as above. 
As Lusztig \cite[Chap.~4]{LuBook} has shown, the $\ba$-function is 
constant on the families of $\Irr(W)$; this is one of the key properties
of this function. In Section~\ref{sec4}, we obtain the following 
refinement:

{\em Let $E,E' \in \Irr(W)$ be such that $E \preceq_L E'$. Then $\ba_{E'}
\leq \ba_E$, with equality if and only if $E,E'$ belong to the same Lusztig 
family.}

We remark that, conjecturally, $\preceq_L$ should coincide with the 
Kazhdan--Lusztig pre-order relation $\leq_{\cLR}$, defined using the 
Kazhdan--Lusztig basis of the associated generic Iwahori--Hecke algebra;
see Remark~\ref{remLR}. It is part of Lusztig's general conjectures in 
\cite[\S 14.2]{Lusztig03} that the $\ba$-function should satisfy a
monotony property as above but with respect to $\leq_{\cLR}$. Thus, 
it is our hope that the results in \cite{klord} and in this paper 
might provide a step towards a proof of Lusztig's conjectures.

%%%%%%%%%%%%%%%%%%%%%%%%%%%%%%%%%%%%%%%%%%%%%%%%%%%%%%%%%%%%%%%%%%%%%%%%%%%
\section{Lusztig's $\ba$-invariants and families} \label{sec2}

Let $W$ be a finite Coxeter group, with generating set $S$ and corresponding 
length function $l\colon W \rightarrow \Z_{\geq 0}$. Let  
$L \colon W \rightarrow \Z$ be a 
weight function, that is, we have $L(ww')=L(w)+L(w')$ whenever $w,w'\in W$ 
are such that $l(ww')=l(w)+l(w')$.  
(This is the setting of Lusztig \cite{Lusztig03}.) 

Throughout this paper, we shall assume that
\begin{center}
\fbox{$\;L(s)\geq 0 \qquad \mbox{for all $s \in S$}.\;$}
\end{center}
If, furthermore, there is some $a\in \Z$ such that $a>0$ and $L(s)=a$
for all $s \in S$, then we say that we are in the {\em equal parameter 
case}.

Let $\Irr(W)$ be the set of (complex) irreducible representations of $W$
(up to isomorphism). Having fixed $L$ as above, we shall
define a function
\[ \Irr(W) \rightarrow \Z_{\geq 0}, \qquad E \mapsto \ba_E.\]
We need one further piece of notation. Let $T=\{wsw^{-1}\mid w
\in W,s\in S\}$ be the set of all reflections in $W$. Let $S'\subseteq S$
be a set of representatives of the conjugacy classes of $W$ which are
contained in $T$. For $s \in S'$, let $N_s$ be the cardinality of the
conjugacy class of $s$; thus, $|T|= \sum_{s\in S'} N_s$. Now let $E \in
\Irr(W)$ and $s\in S'$. Since $s$ has order $2$, it is clear that
$\mbox{trace}(s,E)\in \Z$. Hence, by a well-known result in the character
theory of finite groups, the quantity $N_s\mbox{trace}(s,E)/\dim E$ is an
integer. Thus, we can define
\[ \omega_L(E):=\sum_{s \in S'} \frac{N_s\, \mbox{trace}(s,E)}{\dim E}\,
L(s) \in \Z.\]
(Note that this does not depend on the choice of the set of representatives
$S' \subseteq S$.)

\begin{defn} \label{defafun} We define a function $\Irr(W) \rightarrow
\Z$, $E \mapsto \ba_E$, inductively as follows. If $W=\{1\}$,
then $\Irr(W)$ only consists of the unit representation (denoted 
$1_W$) and we set $\ba_{1_W}:=0$.  Now assume that $W \neq \{1\}$ 
and that the function $M \mapsto \ba_M$ has already been defined for
all proper parabolic subgroups of $W$. Then, for any $E \in \Irr(W)$, we
can define
\[ \ba_E^\prime:= \max\{\ba_M \mid  M \in \Irr(W_J)
\mbox{ where } J \subsetneqq S \mbox{ and } M \uparrow E\}.\]
The notation $M \uparrow E$ is a shorthand for 
$E\mid \mbox{Ind}_{W_J}^W(M)$. 

Finally, we set
\[ \renewcommand{\arraystretch}{1.2} \ba_E:=
\left\{ \begin{array}{cl} \ba_E^\prime & \quad
\mbox{if $\ba_{E\otimes \sgn}^\prime-\ba_E^\prime
\leq \omega_L(E)$},\\\ba_{E \otimes \sgn}^\prime-\omega_L(E) &
\quad \mbox{otherwise},  \end{array}\right.\]
where $\sgn$ denotes the sign representation.
\end{defn}

\begin{rem} \label{defafun1} One immediately checks that the function
$E \mapsto \ba_E$ satisfies the following conditions:
\[ \ba_E \geq \ba_E^\prime \geq 0 \quad \mbox{and}
\quad \ba_{E \otimes \sgn}-\ba_E=\omega_L(E) \quad
\mbox{for all $E \in \Irr(W)$}.\]
This also shows that $\ba_E \geq \ba_M$ if $M \uparrow E$ where 
$M \in \Irr(W_J)$ and $J \subsetneqq S$.
\end{rem}

\begin{exmp} \label{expaasym} (a) If $L(s)=0$  for all $s \in S$, then
$\ba_E=0$ for any $E \in \Irr(W)$.

(b) The unit representation has $\ba$-invariant $0$ and the sign 
representation has $\ba$-invariant $L(w_0)$ where $w_0 \in W$ is the 
longest element; we have $0 \leq \ba_E \leq L(w_0)$ for all $E \in 
\Irr(W)$.  (See \cite[\S 1.3]{GeNico} for details.)
\end{exmp}

We just remark that Lusztig \cite{Lusztig79b}, \cite{Lusztig03} originally 
defined ``$\ba$-invariants'' $\ba_E$ using the ``generic degrees'' of the 
generic Iwahori--Hecke algebra associated with $W,L$ (see 
Remark~\ref{remAa}). The fact that this is equivalent to 
Definition~\ref{defafun} is shown in \cite[Remark~4.3]{mycarp}. The above 
definition will be sufficient for the purposes of this paper.

It will be convenient to introduce the following notation. Let $J 
\subseteq S$, $M \in \Irr(W_J)$ and $E \in \Irr(W)$. Then we write
$M \rightsquigarrow_L E$ if $M \uparrow E$ and $\ba_M=\ba_E$. 

\begin{defn}[Lusztig \protect{\cite[4.2]{LuBook}}] \label{family2}
The partition of $\Irr(W)$ into ``families'' is defined as follows. When
$W= \{1\}$, there is only one family; it consists of the unit representation
of $W$. Assume now that $W \neq \{1\}$ and that families have already been
defined for all proper parabolic subgroups of $W$. Then $E,E' \in \Irr(W)$
are said to be in the same family for $\Irr(W)$ if there exists a sequence
$E= E_0,E_1, \ldots, E_m=E'$ in $\Irr(W)$ such that, for each $i \in \{1,2,
\ldots,m\}$, the following condition is satisfied. There exists a subset
$I_i \subsetneqq S$ and $M_i,M_i' \in \Irr(W_{I_i})$, where $M_i$,
$M_i'$ belong to the same family of $\Irr(W_{I_i})$, such that
either
\begin{align*}
 M_i \rightsquigarrow_L E_{i-1} \qquad &\mbox{and} \qquad 
M_i' \rightsquigarrow_L E_i\\
\intertext{or}
M_i \rightsquigarrow_L E_{i-1} \otimes \sgn \qquad &\mbox{and} \qquad 
M_i' \rightsquigarrow_L E_i\otimes \sgn.
\end{align*}
\end{defn}

Note that it is clear from this definition that, if $\cF\subseteq \Irr(W)$ is
a family, then $\cF \otimes \sgn:=\{E \otimes \sgn \mid E \in \cF\}$ 
is a family, too.

\begin{defn} \label{mydef} Following \cite{klord}, we define a relation 
$\preceq_L$ on $\Irr(W)$ inductively as follows. If $W=\{1\}$, then 
$\Irr(W)$ only consists of the unit representation and this is related to 
itself.  Now assume that $W \neq\{1\}$ and that $\preceq_L$ has already 
been defined for all proper parabolic subgroups of $W$. Let $E,E'\in
\Irr(W)$. Then we write $E\preceq_L E'$ if there is a sequence $E=E_0,
E_1,\ldots, E_m=E'$ in $\Irr(W)$ such that, for each $i \in \{1,2,
\ldots,m\}$, the following condition is satisfied. There exists a 
subset $I_i \subsetneqq S$ and $M_i, M_i'\in \Irr(W_{I_i})$, where 
$M_i \preceq_L M_i'$ within $\Irr(W_{I_i})$, such that either 
\begin{align*}
M_i \uparrow E_{i-1} \qquad &\mbox{and} 
\qquad M_i' \rightsquigarrow_L E_i\\\intertext{or}
M_i\uparrow E_i \otimes \sgn \qquad &
\mbox{and} \qquad M_i' \rightsquigarrow_L E_{i-1}\otimes \sgn.
\end{align*}
\end{defn}

It is already remarked in \cite{klord} that, if $E,E'$ belong to the same 
family, then $E \preceq_L E'$ and $E'\preceq_L E$. In Corollary~\ref{ordfam}
we will see that the converse also holds. Note that this is not clear 
from the definitions.

\begin{rem} \label{remLR} Given the weight function $L$, one can define 
pre-order relations $\leq_{\cL}$, $\leq_{\cR}$, $\leq_{\cLR}$ on $W$, 
using the Kazhdan--Lusztig basis of the associated Iwahori--Hecke 
algebra. The corresponding equivalence classes are called the left, 
right and two-sided cells of $W$; see Lusztig \cite{Lusztig03}. As 
explained in \cite[\S 2]{klord}, the two-sided relation $\leq_{\cLR}$ 
on $W$ induces a pre-order relation on $\Irr(W)$ which we denote by the 
same symbol. It is shown in \cite[Prop.~3.4]{klord} that we have the 
following implication, where $E,E'\in \Irr(W)$:
\[ E\preceq_L E' \qquad \Rightarrow \qquad E \leq_{\cLR} E'.\]
Conjecturally, the reverse implication should also hold. In the equal 
parameter case, this is proved in \cite[Theorem~4.10]{klord}. As far
as unequal parameters are concerned, it is known to be true for any $L$ 
and $W$ of type $F_4$ or $I_2(m)$; see \cite[\S 3]{klord}. In 
Example~\ref{expasym2}, we will see that this also holds for an infinite
collection of weight functions in type $B_n$. A general proof of this 
conjecture, for any weight function $L$, would be a major breakthrough
in this theory.
\end{rem}

\begin{exmp} \label{expordtriv} Assume that $L(s)=0$ for all $s \in S$.
By Example~\ref{expaasym} we have $\ba_E=0$ for all $E \in \Irr(W)$. This 
implies that $E \preceq_L E'$ for all $E,E'\in \Irr(W)$; in particular,
the whole set $\Irr(W)$ is a family.
\end{exmp}

Before we give further (and non-trivial) examples, we discuss two 
reduction statements which will be helpful in the determination of the 
relation $\preceq_L$.  

\begin{rem} \label{ordprod} Let $W=W_1 \times \cdots \times W_k$ be the
decomposition of $W$ into irreducible components. Correspondingly, we
have 
\[ \Irr(W)=\{ E_1 \boxtimes \cdots \boxtimes E_k\mid E_i\in\Irr(W_i)
\mbox{ for $1\leq i\leq k$}\}.\]
Now the function $E \mapsto \ba_E$ is easily seen to be compatible with 
the above decomposition, that is, we have $\ba_{E}=\ba_{E_1}+\cdots+
\ba_{E_k}$ where $\ba_{E_i}$ is defined with respect to the restriction
of $L$ to $W_i$. Furthermore, the induction of representations from parabolic 
subgroups is compatible with the above decomposition. Consequently, the
following hold, where $E_i,E_i'\in \Irr(W_i)$ for $1\leq i \leq k$.
\begin{itemize}
\item[(a)] $E_1 \boxtimes \cdots \boxtimes E_k$ and $E_1' \boxtimes \cdots 
\boxtimes E_k'$ belong to the same family of $\Irr(W)$ if and only if
$E_i$ and $E_i'$ belong to the same family of $\Irr(W_i)$ for $1 \leq i 
\leq k$.
\item[(b)] $E_1 \boxtimes \cdots \boxtimes E_k \preceq_L E_1' \boxtimes 
\cdots \boxtimes E_k'$ within $\Irr(W)$ if and only if $E_i \preceq_L 
E_i'$ within $\Irr(W_i)$ for $1 \leq i \leq k$.
\end{itemize}
Hence, it is sufficient to determine the pre-order relation $\preceq_L$ 
in the case where $(W,S)$ is irreducible. 
\end{rem}

\begin{rem} \label{ordmax} Let $E\in \Irr(W)$ and assume that there 
exists a proper subset $I \subsetneqq S$ and $M\in \Irr(W_I)$ such 
that $M \uparrow E$. Now let $I_1 \subsetneqq S$ such that $I \subseteq 
I_1$. Then, by the transitivity of induction, it is clear that there 
exists some $M_1\in \Irr(W_{I_1})$ such that
\begin{itemize}
\item[(a)] $M \uparrow M_1$ and $M_1 \uparrow E$.
\end{itemize}
Now assume that $M \rightsquigarrow_L E$. Then, by 
Remark~\ref{defafun1} and the transitivity
of induction, there exists some $M_1\in \Irr(W_{I_1})$ such that
\begin{itemize}
\item[(b)] $M \rightsquigarrow_L M_1$ and $M_1 \rightsquigarrow_L E$.
\end{itemize}
These relations immediately imply that, in the formulation of 
Definition~\ref{mydef}, we may assume without loss of generality that 
each $W_{I_i}$ is a maximal parabolic subgroup of $W$. (A similar statement 
concerning families already appeared in \cite[4.2]{LuBook}.)
\end{rem}

\begin{rem} \label{afuncsum} Assume that $(W,S)$ is irreducible and 
that $L(s)>0$ for all $s \in S$. The invariants $\ba_E$ and the 
families are explicitly known in all cases by the work of Lusztig; 
see \cite[Chap.~4]{LuBook}, \cite[\S 6.5]{gepf} (for the equal parameter 
case) and \cite[Chap.~23]{Lusztig03} (for the remaining cases). 
The article \cite{klord} explicitly describes the pre-order relation 
$\preceq_{L}$ on $\Irr(W)$, except for the case where $W$ is of
type $B_n$ and we are not in the equal parameter case. 
\end{rem}

\begin{exmp} \label{ordA} Let $n \geq 1$ and $(W,S)$ be of type $A_{n-1}$. 
Then $W$ can be identified with the symmetric group $\fS_n$ and $\Irr(W)$ 
is parametrised by the partitions of $n$. Thus, we can write
\[ \Irr(W)=\{E^\lambda \mid \lambda \mbox{ is a partition of } n\}.\]
For example, the unit representation is labelled by $(n)$ and the sign
representation is labelled by $(1^n)$. Assume now that $L$ is not 
identically zero. All generators in $S$ are conjugate in $W$ and so there 
is some $a\in \Z_{>0}$ such that $L(s)=a$ for all $s \in S$. 

Let $\lambda$ be a partition of $n$. Writing the parts of $\lambda$
as $\lambda_1\geq \lambda_2 \geq \ldots \geq \lambda_N \geq 0$ where 
$N \geq 1$, we set 
\[ n(\lambda):=\sum_{1\leq i \leq N} (i-1)\lambda_i.\]
Then we have $\ba_{E^\lambda}=n(\lambda)a$; see, for example, \cite[6.5.8, 
9.4.5]{gepf}. By Lusztig \cite[4.4]{LuBook}, each family of $\Irr(W)$ 
consists of a single representation. Now recall that the dominance
order $\trianglelefteq$ on partitions is defined as follows. Let $\lambda, 
\mu$ be two partitions of $n$. By adding zeroes if necessary, we can write
\[ \lambda=(\lambda_1\geq\lambda_2\geq\ldots \geq \lambda_N \geq 0)\qquad 
\mbox{and}\qquad \mu=(\mu_1\geq \mu_2 \geq \ldots \geq \mu_N \geq 0)\]
for some $N \geq 1$. Then we have
\[\lambda\trianglelefteq \mu\quad \stackrel{\text{def}}{\Leftrightarrow}
\quad \sum_{1\leq i \leq d} \lambda_i \leq \sum_{1\leq i \leq d} \mu_i
\quad \mbox{for all $d \in \{1,\ldots,N\}$}.\]
With this notation, the following three conditions are equivalent:
\begin{itemize}
\item[(a)] $E^\lambda \preceq_L E^\mu$.
\item[(b)] $\lambda \trianglelefteq \mu$.
\item[(c)] There exists a subset $I \subseteq S$ such that $\sgn_I 
\uparrow E^\lambda$ and $\sgn_I \rightsquigarrow_L E^\mu$.
\end{itemize}
\begin{proof}
``(a) $\Rightarrow$ (b)'' By Remark~\ref{remLR}, the assumption implies 
that $E^\lambda \leq_{\cLR} E^\mu$. Then (b) holds by 
\cite[Theorem~5.1]{mymurphy}; see also \cite[\S 2.8]{GeNico}. (Note that 
no geometric interpretation is required here.)

``(b) $\Rightarrow$ (c)'' Let $W_I\subseteq W$ be the parabolic subgroup 
which corresponds, under the identification $W \cong \fS_n$, to the Young
subgroup $\fS_{\overline{\mu}}$ where $\overline{\mu}$ denotes the 
conjugate partition. Then it is well-known that $\sgn_I \rightsquigarrow_L 
E^\mu$; see, for example, \cite[Prop.~5.4.7]{gepf}. So it remains to show 
that $\sgn_I \uparrow E^\lambda$.  Now, since $\lambda \trianglelefteq 
\mu$, we also have $\overline{\mu} \trianglelefteq \overline{\lambda}$; 
see \cite[I.1.11]{Mac}. Consequently, the Kostka number 
$\kappa_{\overline{\lambda}, \overline{\mu}}$ is non-zero; see 
\cite[I.6.4]{Mac}. This implies that $1_I \uparrow 
E^{\overline{\lambda}}$ where $1_I$ stands for the unit representation; 
see the remark following \cite[I.7.8]{Mac}. Now note that, by 
\cite[Cor.~5.4.9]{gepf}, we have $E^\lambda=E^{\overline{\lambda}}
\otimes \sgn$; furthermore, $\Ind_I^S(\sgn_I)=\Ind_I^S(1_I)\otimes\sgn$. 
Consequently, we also have $\sgn_I \uparrow E^\lambda$, as required.

``(c) $\Rightarrow$ (a)'' This is clear by the definition of $\preceq_L$.
\end{proof}
\end{exmp}

\begin{exmp} \label{bnsetup} Let $n \geq 1$ and $W=W_n$ be a Coxeter
group of type $B_n$, with generators and diagram given by 
\begin{center}
\begin{picture}(250,20)
\put( 10, 5){$B_n$}
\put( 63,13){$t$}
\put( 91,13){$s_1$}
\put(121,13){$s_2$}
\put(201,13){$s_{n{-}1}$}
\put( 65, 5){\circle*{5}}
\put( 95, 5){\circle*{5}}
\put(125, 5){\circle*{5}}
\put( 65, 7){\line(1,0){30}}
\put( 65, 3){\line(1,0){30}}
\put( 95, 5){\line(1,0){30}}
\put(125, 5){\line(1,0){20}}
\put(155,5){\circle*{1}}
\put(165,5){\circle*{1}}
\put(175,5){\circle*{1}}
\put(185, 5){\line(1,0){20}}
\put(205, 5){\circle*{5}}
\end{picture}
\end{center}
Then $\Irr(W_n)$ is naturally parametrised by bipartitions of $n$, 
that is, pairs of partitions $(\lambda,\mu)$ such that $|\lambda|+ 
|\mu|=n$; we shall write
\[ \Irr(W_n)=\{E^{(\lambda,\mu)} \mid (\lambda,\mu) \mbox{ is a 
bipartition of } n \}.\]
For example, the unit and the sign representation are labelled by 
$((n), \varnothing)$ and $(\varnothing,(1^n))$, respectively; see 
\cite[\S 5.5]{gepf}. A weight function $L \colon W_n \rightarrow \Z$ 
is specified by the two parameters 
\[ b:=L(t) \geq 0 \qquad \mbox{and} \qquad a:=L(s_i) \geq 0 \quad
\mbox{for $1\leq i \leq n-1$}.\]
Thus, the relation $\preceq_L$ will depend on $a,b$; simple examples 
show that $\preceq_L$ is really different for different values of $a,b$.
\end{exmp}

In the next section, we will begin with the study of this case by 
considering certain pre-order relations $\preceq_{a,b}$ on bipartitions. 
These will turn out to be the key for describing the relation $\preceq_L$ 
on $\Irr(W_n)$; see Theorem~\ref{mainbn}. 

%%%%%%%%%%%%%%%%%%%%%%%%%%%%%%%%%%%%%%%%%%%%%%%%%%%%%%%%%%%%%%%%%%%%%%%%%%%
\section{Ordering bipartitions} \label{sec1}

The aim of this section is to define suitable generalisations of the
dominance order for partitions to the setting of bipartitions of $n$. 
The framework for doing this is provided by the combinatorics developed 
by Lusztig in \cite[Chap.~22]{Lusztig03}.

Let us fix some notation.  As in Example~\ref{bnsetup}, we shall fix two 
integers $a,b \in \Z$ such that $a>0$ and $b \geq 0$. (The case
where $a=0$ will be treated separately; see Example~\ref{bnanull}.)
Division with remainder defines then two integers $r, b' \geq 0$  
by \[ b=ra+b'\quad\text{and}\;\; 0\leq b'<b\, .\]
Following Lusztig \cite[22.6]{Lusztig03}, given any integer $N \geq 0$, we 
define $\cM_{a,b;n}^N$ to be the set of all multisets $Z=\{z_1,z_2,\ldots,
z_{2N+r}\}$ such that the following hold:
\begin{itemize}
\item[(M1)] The entries of $Z$ are elements of $\Z_{\geq 0}$ and we have 
\[ \sum_{1\leq i \leq 2N+r} z_i=na+N^2a+N(b-a)+\binom{r}{2}a+rb'.\]
\item[(M2)] If $b'=0$, there are at least $N+r$ distinct entries in $Z$, no 
entry is repeated more than twice, and all entries of $Z$ are contained 
in $\Z a$.
\item[(M3)] If $b'>0$, then all entries of $Z$ are distinct; furthermore, 
$N$ entries of $Z$ are contained in $\Z a$ and $N+r$ entries are contained 
in $b'+\Z a$.
\end{itemize}
There is a ``shift'' operation $\cM_{a,b;n}^N \rightarrow
\cM_{a,b;n}^{N+1}$ given by 
\[ \{z_1,z_2,\ldots,z_{2N+r}\} \mapsto \{ 0,b',z_1+a,z_2+a,\ldots,
z_{2N+r}+a\}.\]
Given two multisets $Z \in \cM_{a,b;n}^{N}$ and $Z' \in 
\cM_{a,b;n}^{N'}$ (where $N,N'\geq 0$ are integers), we write 
$Z \sim Z'$ if one of $Z,Z'$ can be obtained from the other by 
a finite sequence of shift operations. This defines an equivalence
relation on the union of all multisets $\bigcup_{N \geq 0} \cM_{a,b;n}^N$.
These constructions are related to bipartitions of $n$, in the 
following way. 

%Let $\cM_{a,b;n}$ denote the set of equivalence classes; given $Z\in
%\cM_{a,b;n}^N$ (for some $N \geq 0$), we denote the corresponding 
%equivalence class by $[Z]\in \cM_{a,b;n}$.

\begin{defn}[Lusztig] \label{def0}  Let $(\lambda,\mu)$ be a 
bipartition of $n$. Choosing a sufficiently large integer $N\geq 0$ and 
adding zeroes if necessary, we can write 
\[ \lambda=(\lambda_1\geq \lambda_2 \geq \ldots \geq \lambda_{N+r} \geq 0)
\quad \mbox{and}\quad  \mu=(\mu_1\geq \mu_2 \geq \ldots \geq 
\mu_N\geq 0).\] 
Then we define $Z_{a,b}^N(\lambda,\mu)$ to be the multiset formed by the 
$2N+r$ entries 
\begin{alignat*}{2}
(&\lambda_{i}+N+r-i)a+b' &&\qquad (1\leq i \leq N+r),\\
(&\mu_{j}+N-j)a &&\qquad (1 \leq j \leq N). 
\end{alignat*}
As pointed out in \cite[22.10]{Lusztig03}, we have $Z_{a,b}^N(\lambda,
\mu)\in \cM_{a,b;n}^N$. Furthermore, if we choose another $N'$, then 
$Z_{a,b}^N(\lambda,\mu) \sim Z_{a,b}^{N'}(\lambda,\mu)$. 

The above array of $2N+r$ integers, arranged in two rows, is 
an example of Lusztig's ``symbols'', which originally appeared in the
classification of unipotent representations of finite classical
groups; see \cite[Chap.~4]{LuBook}. 
\end{defn}

\begin{exmp} \label{expmult1} Let us consider the bipartition $(4311,32)$ of
$14$. 

First assume that $a=b=1$. Then $r=1$ and $b'=0$. We can take $N=3$ and 
write our bipartition as $(4311,320)$. The corresponding multiset is
\[ Z_{1,1}^3(4311,32)=\{7,5,5,3,2,1,0\}.\]
If we take $N=5$, then we write our bipartition as $(431100,32000)$ and
the corresponding multiset is 
\[ Z_{1,1}^5(4311,32)=\{9,7,7,5,4,3,2,1,1,0,0\}.\]
This multiset is obtained from the previous one by performing two  shifts.

Now assume that $a=2$, $b=1$. Then $r=0$ and $b'=1$. We can take $N=4$ and 
write our bipartition as $(4311,3200)$. The corresponding multiset is
\[ Z_{2,1}^4(4311,32)=\{15,12,11,8,5,3,2,0\}.\]
\end{exmp}

\begin{rem} \label{rem11} 
Let $\gamma$ be the right hand side of the formula in (M1). If we 
arrange the entries of a multiset $Z\in \cM_{a,b;n}^N$ in decreasing order, 
we obtain a partition of~$\gamma$. Thus, all multisets in $\cM_{a,b;n}^N$ 
can be regarded as partitions of $\gamma$. Consequently, the set 
$\cM_{a,b;n}^N$ is partially ordered by the dominance order and it makes 
sense to write $Z\trianglelefteq Z'$ for $Z,Z'\in\cM_{a,b;n}^N$.

One easily checks that if $Z_1,Z_2 \in \cM_{a,b;n}^{N}$ and 
$Z_1',Z_2'\in \cM_{a,b;n}^{N'}$ (where $N,N'\geq 0$ are integers),
then the following implication holds:
\[ Z_1 \sim Z_1', \quad Z_2' \sim Z_2' \quad \mbox{and} \quad Z_1 
\trianglelefteq Z_2\qquad\Rightarrow\qquad Z_1'\trianglelefteq Z_2'.\]
Thus, $\trianglelefteq$ can be regarded as a partial order on the 
equivalence classes of multisets as above, modulo the shift operation. 
\end{rem}

We are now ready to define the desired generalisations of the dominance
order.

\begin{defn} \label{def1} We define a pre-order relation $\preceq_{a,b}$ on 
the set of bipartitions of $n$, as follows. Let $(\lambda,\mu)$ and
$(\lambda', \mu')$ be bipartitions of $n$. Choose a sufficiently large 
integer $N \geq 0$ and consider the multisets $Z_{a,b}^N(\lambda,\mu)$, 
$Z_{a,b}^N(\lambda',\mu')$ in Definition~\ref{def0}. Then 
\[ (\lambda,\mu)\preceq_{a,b} (\lambda',\mu') \qquad 
\stackrel{\text{def}}{\Leftrightarrow} \qquad Z_{a,b}^N(\lambda,\mu)
\trianglelefteq Z_{a,b}^N(\lambda',\mu').\]
We say that the bipartitions $(\lambda,\mu)$ and $(\lambda',\mu')$ 
belong to the same ``combinatorial family'' if $Z_{a,b}^N(\lambda,\mu)=
Z_{a,b}^N(\lambda',\mu')$. Then $\preceq_{a,b}$ induces a partial order
on the set of combinatorial families which we denote by the same symbol. 
Note that these definitions do not depend on the choice of $N$. 
\end{defn}

(In a somewhat different context, similar pre-orders on pairs of 
partitions are considered in \cite[\S 5.5]{GeNico} and also 
by Chlouveraki et al.  \cite[\S 5]{MG}.)

\begin{rem} \label{expsing} Assume that $b'>0$. Let $(\lambda,\mu)$ be
a bipartition of $n$ and $Z_{a,b}^N(\lambda,\mu)$ be the corresponding
multiset (for some $N \geq 0$). Then the expressions in 
Definition~\ref{def0} show that $(\lambda,\mu)$ is uniquely determined 
by $Z_{a,b}^N(\lambda,\mu)$. Consequently, all combinatorial families 
are singleton sets in this case, and $\preceq_{a,b}$ is a partial order.
\end{rem}

On the other hand, if $b'=0$, then simple examples show that, in 
general, the combinatorial families will contain more than one element. 

Also note that if $b'=0$, then the following equivalence holds
\[ (\lambda,\mu) \preceq_{a,b} (\lambda',\mu') \qquad \Leftrightarrow 
\qquad (\lambda,\mu) \preceq_{1,r} (\lambda',\mu')\]
for all bipartitions $(\lambda,\mu)$ and $(\lambda',\mu')$ of $n$.

\begin{exmp} \label{expsubasy} Assume that  $a=1$ and $b=n-1$; 
we call this the ``sub-asymptotic'' case. We have $r=n-1$ and $b'=0$. Let 
us describe the combinatorial families in this case. 
Let $\cF_0:=\{(1^k, l)\mid k,l\geq 0\;\text{and}\; k+l=n\}$. We claim that:
\begin{itemize}
\item[(i)] $\cF_0$ is a combinatorial family;
\item[(ii)] all the other combinatorial families are singleton sets.
\end{itemize}
Indeed, let $(\lambda,\mu)=(1^k,(l))$ where $n=k+l$. Let $N:=n$. The 
corresponding multiset $Z_{a,b}^n(\lambda,\mu)$ contains the entries
\begin{alignat*}{2}
\lambda_i+N+r-i&=\lambda_i+2n-1-i && \qquad (1\leq i \leq 2n-1),\\
\mu_j+N-j&=\mu_j+n-j && \qquad (1\leq j \leq n).
\end{alignat*}
The first row of the above array yields the entries $\{0,1,2,\ldots,
2n-1\} \setminus \{2n-k-1\}$; the second row yields the entries 
$\{0,1,2,\ldots, n-2\} \cup\{l+n-1\}$. Thus, since $l+n-1=2n-k-1$,
we obtain
\begin{align*}
Z_{a,b}^n(\lambda,\mu)&=\{0,1,2,\ldots,2n-1\} \cup \{0,1,2,\ldots,n-2\}\\
&=\{ 0,0,1,1,2,2,\ldots,n-2,n-2,n-1,n,n+1,\ldots,2n-1\}.
\end{align*}
Hence, all bipartitions of the form $((1^k),(l))$ (where $n=k+l$)
belong to the same combinatorial family. Conversely, let $(\lambda,\mu)$
be a bipartition which is not of this form. Suppose first that 
$\lambda_1>1$. Then $\lambda_1+N+r-1>2n-1$ and so $(\lambda,\mu)$ is not
in the combinatorial family of $\cF_0$. Similarly, if $\mu$ has at least 
two (non-zero) parts, then the sequence $\{\mu_j+N-j \mid 1 \leq j \leq N\}$
will not contain all of the numbers $0,1,2,\ldots,n-2$ and, hence, these 
numbers will not all be repeated twice in $Z_{a,b}^n(\lambda,\mu)$. So, 
again, $(\lambda, \mu)$ is not in the combinatorial family of $\cF_0$. 
Thus, (i) is proved. The proof of (ii) is a similar combinatorial 
exercise; the precise argument is along the lines of the discussion in 
Section~\ref{secasym}. We omit further details.
\end{exmp}

\begin{exmp} \label{expab1} Assume that $a=1$ and $b \in 
\{0,1\}$. Thus, we are either in the equal parameter case or in the
case which is relevant to groups of type $D_n$ (see Example~\ref{expdn}). 
We have $b'=0$ and $r \in \{0,1\}$. 

Let $(\lambda,\mu)$ be a bipartition of $n$ and write
\[\lambda=(\lambda_1\geq \lambda_2\geq \cdots \geq \lambda_{N+r} \geq 0),
\qquad  \mu=(\mu_1\geq \mu_2\geq \ldots \geq \mu_N \geq 0)\]
for some $N \geq 0$. As in \cite[Chap.~4]{LuBook}, we say that
$(\lambda,\mu)$ is ``special''  (or ``$(a,b)$-special'' if $a,b$ are not
clear from the context), if the following conditions hold:
\begin{alignat*}{2}
\lambda_i+1 \geq \mu_i \geq \lambda_{i+1} & \quad (1\leq i \leq N) &&
\qquad \mbox{if $a=b=1$},\\
\lambda_N\geq \mu_N \mbox{ and } \lambda_i \geq \mu_i \geq 
\lambda_{i+1}-1 & \quad (1\leq i \leq N-1) && 
\qquad \mbox{if $a=1$, $b=0$}.
\end{alignat*}
Then every combinatorial family contains a unique special bipartition. 
\end{exmp}

%Furthermore, one notices that, if $(\lambda,\mu)$ and $(\lambda',\mu')$ 
%are special bipartitions, then the definition of $\preceq_{1,b}$ can be 
%rewritten as follows:
%\[(\lambda,\mu) \preceq_{1,b} (\lambda',\mu')\quad 
%\Leftrightarrow \quad \left\{\begin{array}{rcl} 
%\displaystyle \sum_{1\leq i\leq d} (\lambda_i+\mu_i) &\leq &\displaystyle 
%\sum_{1\leq i\leq d} (\lambda_i'+\mu_i')\\ \displaystyle 
%\lambda_d+ \sum_{1\leq i<d} (\lambda_i+\mu_i) &\leq &\displaystyle 
%\lambda_d'+ \sum_{1\leq i<d} (\lambda_i'+\mu_i') \\ 
%\multicolumn{2}{r}{\mbox{(for all $d \geq 1$)}} &\end{array} \right.\]
%where $\lambda=(\lambda_1\geq \lambda_2 \geq \ldots \geq 0)$,
%$\lambda'= (\lambda_1' \geq \lambda_2' \geq \ldots \geq \ 0)$,
%$\mu=(\mu_1\geq \mu_2 \geq \ldots \geq 0)$ and $\mu'=(\mu_1'\geq
%\mu_2'\geq \ldots \geq 0)$. This formulation is due to Spaltenstein 
%\cite[\S 4]{Spalt}.

Next recall from Example~\ref{ordA} the definition of the 
numerical invariant $n(\lambda)$ associated with a partition $\lambda$
of $n$. This invariant has the following compatibility property with the 
dominance order: Let $\lambda$ and $\mu$ be partitions of $n$ such that 
$\lambda \trianglelefteq \mu$. Then $n(\mu)\leq n(\lambda)$, with equality 
only for $\lambda=\mu$; see \cite[Exc.~5.6]{gepf}. We will now see that 
the $\ba$-invariants attached to bipartitions by Lusztig 
\cite[\S 22]{Lusztig03} satisfy a similar compatibility property.

\begin{defn} \label{def2} Let $N \geq 0$ and $Z \in \cM_{a,b;n}^N$. 
Let $Z^0$ be the unique multiset in $\cM_{a,b;0}^N$; its  entries are
\[ 0,\,a,\,2a,\,\ldots,\,(N-1)a,\,b',\,a+b',\,2a+b',\,\ldots,\,(N+r-1)a+b'.\]
(If $N=0$, then these entries are $ia+b'$ for $0 \leq i \leq r-1$.)
Then we define 
\[ \ba_{a,b}(Z):=\sum_{1\leq i \leq 2N+r} (i-1)z_i-
\sum_{1\leq i\leq 2N+r} (i-1)z_i^0,\]
where $z_1,z_2,\ldots,z_{2N+r}$ are the entries of $Z$ and $z_1^0,z_2^0,
\ldots,z_{2N+r}^0$ are the entries of $Z^0$, both arranged in decreasing 
order. 

Now let $(\lambda,\mu)$ be a bipartition on $n$. Choose $N$ sufficiently 
large and consider the multiset $Z_{a,b}^N(\lambda,\mu)$ in 
Definition~\ref{def0}. Then we define 
\[ \ba_{a,b}(\lambda,\mu):=\ba_{a,b}\bigl(Z_{a,b}^N(\lambda,\mu)\bigr).\]
\end{defn}

\begin{rem}  \label{lem12} (i)
Note that the above formula coincides with that given by Lusztig 
\cite[Prop.~22.14]{Lusztig03}; in particular, $\ba_{a,b}(\lambda,\mu)$ 
does not depend on the choice of $N$. 

%\label{lem12} 
(ii) Let $N \geq 0$ and $Z,Z' \in \cM_{a,b;n}^N$ be 
such that  $Z \trianglelefteq Z'$. Then $\ba_{a,b}(Z')\leq \ba_{a,b}(Z)$,
with equality only if $Z=Z'$. (This follows by an argument entirely analogous to that
for partitions and the invariants $n(\lambda)$; see \cite[Exc.~5.6]{gepf}.)

(iii) Consequently, if $(\lambda,\mu)$ and $(\lambda',\mu')$ are bipartitions 
of $n$ such that $(\lambda,\mu) \preceq_{a,b} (\lambda',\mu')$, then 
$\ba_{a,b}(\lambda',\mu')\leq \ba_{a,b}(\lambda,\mu)$, with equality only 
if $(\lambda,\mu)$ and $(\lambda',\mu')$ belong to the same combinatorial 
family. 
\end{rem}

Finally, we come to the most subtle property of the pre-order relation
$\preceq_{a,b}$. Recall that, if $\lambda$ and $\mu$ are partitions of 
$n$ such that $\lambda \trianglelefteq \mu$, then we also have 
$\overline{\mu} \trianglelefteq \overline{\lambda}$; see 
\cite[I.1.11]{Mac}. Here, for any partition $\nu$, we denote by 
$\overline{\nu}$ the conjugate partition. Our task is to generalise this 
to bipartitions. Following Lusztig \cite[22.8]{Lusztig03}, we can define 
a conjugation operation on multisets as follows. Let $Z=\{z_1,z_2, 
\ldots, z_{2N+r}\}\in\cM_{a,b;n}^N$. Let $t \geq 0$ be an integer which 
is large enough such that the multiset 
\[ \{ta+b'-z_1,ta+b'-z_2,\ldots,ta+b'-z_{2N+r}\}\]
is contained in the multiset 
\[ \{0,a,2a,\ldots,ta,b',a+b',2a+b', \ldots,ta+b'\}.\]
Then we define $\overline{Z}$ to be the complement of the first of the 
above two multisets in the second. As pointed out in \cite[22.8]{Lusztig03}, 
we have 
\[ \overline{Z}\in \cM_{a,b;n}^{t+1-N-r}.\] 
(Note that (i) $\overline{\overline{Z}}\sim Z$ and (ii) $\overline{Z}_{(t)}
\sim\overline{Z}_{(t')}$, where $\overline{Z}_{(t)}$ and $\overline{Z}_{(t')}$
denote the conjugates of $Z$ calculated for the integers $t$ and $t'$
respectively.)

Also note the following interpretation in terms of bipartitions.

\begin{lem} \label{lem11a} Let $(\lambda,\mu)$ be a bipartition of $n$. 
Then $\overline{Z_{a,b}^N(\lambda,\mu)} \sim Z_{a,b}^N(\overline{\mu},
\overline{\lambda})$.
\end{lem}

\begin{proof} This follows from \cite[22.18]{Lusztig03} and 
\cite[5.5.6]{gepf}.
\end{proof}

We can now state the following fundamental compatibility property. 

\begin{lem} \label{lem11} In the above setting, let $Z,Z' \in \cM_{a,b;n}^N$
be such that $Z \trianglelefteq Z'$. Let $t$ be such that both 
$\overline{Z}$ and $\overline{Z}'$ are defined and in $\cM_{a,b;n}^{t+1-
N-r}$. Then we have $\overline{Z}' \trianglelefteq \overline{Z}$. 
\end{lem}

\begin{proof} For the proof we distinguish the two cases where
either $b'=0$ or $b'>0$. The details of the argument will be given
in Section~\ref{secproof1}. 
\end{proof}

As an immediate consequence we obtain the following key property of
the relation $\preceq_{a,b}$. This is one of the crucial ingredients
in the proof of Theorem~\ref{mainbn}.

\begin{prop} \label{cor11} Let $(\lambda,\mu)$ and $(\lambda',\mu')$ be 
bipartitions of $n$. Then 
\[ (\lambda,\mu)\preceq_{a,b} (\lambda',\mu') \qquad \Leftrightarrow 
\qquad (\overline{\mu}',\overline{\lambda}')\preceq_{a,b} (\overline{\mu},
\overline{\lambda}).\]
\end{prop}

\begin{proof} This is clear by the definition of $\preceq_{a,b}$ and
Lemmas~\ref{lem11a} and \ref{lem11}. \end{proof}

%%%%%%%%%%%%%%%%%%%%%%%%%%%%%%%%%%%%%%%%%%%%%%%%%%%%%%%%%%%%%%%%%%%%%%%%%%%
\section{Proof of Lemma~\ref{lem11}: the case where $b'=0$} 
\label{secproof1}

Throughout this section we assume that $b'=0$. By Remark~\ref{expsing},
we have
\[(\lambda,\mu) \preceq_{a,b} (\lambda',\mu')\quad 
\Leftrightarrow \quad (\lambda,\mu) \preceq_{1,r} (\lambda',\mu')\]
for all bipartitions $(\lambda,\mu)$ and $(\lambda',\mu')$ of $n$.
Thus, in order to prove Lemma~\ref{lem11}, we
can assume without loss of generality that $a=1$ and $b=r
\geq 0$. Then $\cM_{1,r;n}^N$ consists of multisets whose entries are
non-negative integers such that the conditions (M1) and (M2) in
Section~\ref{sec1} hold.

It will now be convenient to work with
a slightly larger class of multisets. We define $\tilde{\cM}_{1,r;n}^N$
to be the set of all multisets whose entries are non-negative integers
satisfying the condition (M1) and such that, instead of (M2), the
weaker condition that no entry is repeated more than twice holds (but
we do not make an assumption on the number of distinct entries). Thus,
$\cM_{a,b;n}^N \subseteq \widetilde{\cM}_{a,b;n}^N$. We shall prove the 
statement in Lemma~\ref{lem11} for this larger class of multisets. 
The advantage of using $\widetilde{\cM}_{a,b;n}^N$ instead of $\cM_{a,b;n}^N$ 
is that we have the following simple description of adjacent multisets in 
$\widetilde{\cM}_{a,b;n}^N$ with respect to $\trianglelefteq$.

\begin{lem} \label{adja1} Let $Z,Z' \in \widetilde{\cM}_{a,b;n}^N$ be 
such that $Z \trianglelefteq Z'$ and $Z\neq Z'$. Assume that $Z,Z'$ are
adjacent with respect to $\trianglelefteq$ (that is, if $Z\trianglelefteq
Z''\trianglelefteq Z'$ for some $Z''\in\widetilde{\cM}_{a,b;n}^N$, then $Z=
Z''$ or $Z''=Z'$). Write $Z=\{z_1\geq z_2 \geq \ldots \geq z_{2N+r}\}$ 
and $Z'=\{z_1'\geq z_2'\geq \ldots \geq z_{2N+r}'\}$. Then there exist 
indices $1\leq k<l\leq 2N+r$ such that $z_k'=z_k+1$, $z_l'=z_l-1$ and 
$z_i'=z_i$ for all $i\neq k,l$.
\end{lem}

\begin{proof} This is similar to the proof of the analogous result
for general partitions in \cite[I.1.16]{Mac}.

Let $k \geq 1$ be such that 
$z_i=z_i'$ for $1 \leq i \leq k-1$ and $z_k<z_k'$. Note that, if
$k \geq 2$, then this implies $z_k<z_k'\leq z_{k-1}'=z_{k-1}$. 
Furthermore, we have $z_1+\ldots +z_k < z_1' + \ldots +z_k'$. Let
$l>k$ be minimal such that 
\[ z_1+\ldots +z_l=z_1'+\ldots + z_l'.\]
Then we have $z_1+\ldots +z_{l-1}<z_1'+\ldots + z_{l-1}'$ and so
$z_l >z_l'$. Since $z_1+\ldots +z_{l+1}\leq z_1'+\ldots + z_{l+1}'$,
we also have $z_{l+1}\leq z_{l+1}'$ and so
\[ z_l>z_l'\geq z_{l+1}'\geq z_{l+1}.\]
We define a multiset $Z''=\{z_1'',\ldots,z_{2N+r}''\}$ by $z_k'':=
z_k+1$, $z_l'':=z_l-1$ and $z_i'':=z_i$ for all $i \neq k,l$. Note 
that $z_1''\geq z_2''\geq \ldots \geq z_{2N+r}''$. We claim that 
$Z'' \in \widetilde{\cM}_{a,b;n}^N$. 

All that needs to
be checked is that no entry of $Z''$ is repeated more than twice.
Note that it could happen that $k \geq 3$ and $z_{k-2}=z_{k-1}=z_k+1$,
in which case the entry $z_{k-1}$ would be repeated three times in
$Z''$. Similarly, it could happen that $l+2\leq 2N+r$ and $z_l-1=
z_{l+1}=z_{l+2}$, in which case the entry $z_{l+1}$ would be repeated 
three times in $Z''$. Hence, we must show that these two situations 
cannot occur.

First assume, if possible, that $k \geq 3$ and $z_{k-2}=z_{k-1}=z_k+1$.
Now, we have $z_{k-2}'=z_{k-2}=z_{k-1}=z_{k-1}'$. Since no entry
of $Z'$ is repeated more than twice, this implies $z_k'<z_{k-1}'$.
But, we have $z_k<z_k'$ and so $z_k+1<z_k'+1\leq z_{k-1}'=z_{k-1}$, 
contradicting our assumption.
Next assume, if possible, that $l+2\leq 2N+r$ and $z_l-1=z_{l+1}=
z_{l+2}$. Since $z_l>z_l'\geq z_{l+1}'\geq z_{l+1}$, this would imply 
$z_l'= z_{l+1}'=z_{l+1}$. Since $z_1+\ldots + z_{l+2}\leq z_1'+\ldots 
+ z_{l+2}'$, we have $z_{l+2}\leq z_{l+2}'$. This yields $z_{l+1}'=
z_{l+1}=z_{l+2} \leq z_{l+2}'$ and so $z_l'=z_{l+1}'= z_{l+2}'$, 
contradicting the fact that no entry of $Z'$ is repeated more than twice. 
This shows that $Z'' \in \widetilde{\cM}_{a,b;n}^N$. We clearly
have $Z \trianglelefteq Z'' \trianglelefteq Z'$.  Since $Z,Z'$ are
assumed to be adjacent, we conclude that $Z'=Z''$, as required.

\end{proof}

\begin{rem} \label{remadja} With somewhat more effort, one can
show that the statement of Lemma~\ref{adja1} holds for adjacent 
multisets in $\cM_{1,r;n}^N$, but we shall not need this here.
\end{rem}

We can now complete the proof of Lemma~\ref{lem11} in this case, 
as follows. As already mentioned, we shall prove the desired implication
for the larger class of multisets $\widetilde{\cM}_{a,b;n}^N$. So let $Z,Z'
\in \widetilde{\cM}_{a,b;n}^N$ be such that $Z \trianglelefteq Z'$. Let $t$ 
be a sufficiently large integer such that both $\overline{Z}$ and 
$\overline{Z}'$ are defined and in $\widetilde{\cM}_{a,b;n}^{t+1-N-r}$; see 
Section~\ref{sec1}. Note that this definition does indeed work for 
multisets in $\widetilde{\cM}_{a,b; n}^N$.) Thus, $\overline{Z}$ and 
$\overline{Z}'$ are obtained by taking the complements of the multisets 
\[ \{t-z \mid z \in Z\} \qquad \mbox{and} \qquad 
\{t-z' \mid z' \in Z'\}\] 
in the multiset $\{0,0,1,1,2,2,\ldots,t,t\}$. We must show
that $\overline{Z}'\trianglelefteq \overline{Z}$. For this purpose,
we can assume without loss of generality that $Z,Z'$ are adjacent 
with respect to $\trianglelefteq$. Then Lemma~\ref{adja1} shows
that $Z$ is obtained from $Z'$ by increasing a suitable entry by $1$ 
and decreasing another suitable entry by $1$. A similar statement
then also holds for $\overline{Z}$ and $\overline{Z}'$ and one
immediately checks that $\overline{Z}' \trianglelefteq \overline{Z}$. 
Thus, Lemma~\ref{lem11} holds in the case where $b'=0$.

%%%%%%%%%%%%%%%%%%%%%%%%%%%%%%%%%%%%%%%%%%%%%%%%%%%%%%%%%%%%%%%%%%%%%%%%%%%%%%%%
\section{Proof of Lemma~\ref{lem11}: The case where $b'>0$}
\label{secproof2}

Throughout this section we assume that $b'>0$. An essential feature
of this case is that then all entries in a multiset in
$\cM_{a,b;n}^N$ are distinct; that is, we are dealing with actual
sets (finite subsets of $\Z_{\geq 0}$) and not just with multisets. 
Now it might be possible to use a
similar argument as in the previous section, but the following example
illustrates that adjacent pairs with respect to $\trianglelefteq$
certainly are more difficult to describe in this case.

\begin{exmp} \label{badexp}
Let $n=N=2$, $a=2$ and $b=3$. We have:
\[\begin{array}{cc} \hline (\lambda,\mu) & Z_{2,3;2}^2(\lambda,\mu)\\
\hline
(\varnothing,11)    & \{ 5, 4, 3, 2, 1 \}\\
(\varnothing,2)     & \{ 6, 5, 3, 1, 0 \}\\ 
(1,1)               & \{ 7, 4, 3, 1, 0 \}\\
(11,\varnothing)    & \{ 7, 5, 2, 1, 0 \}\\
(2,\varnothing)     & \{ 9, 3, 2, 1, 0 \}\\
\hline \end{array}\]
In this table, bipartitions in consecutive rows are adjacent with
respect to $\trianglelefteq$ (and these are all the adjacent pairs).
Thus, we see that adjacent sets $Z,Z' \in \cM_{2,3;2}^2$ can differ
in more than $2$ entries. (This is not an isolated example.)
\end{exmp}

Because of this difficulty, we follow an alternative route for proving
Lemma~\ref{lem11}.

It will be useful to introduce the following notation concerning
finite subsets of $\Z_{\geq 0}$. If $X$ is such a subset, we denote
by $\# X$ the number of elements of $X$ and by $X^+$ the sum of the
entries of $X$. Let $M \geq 0$ be an integer such that $x\leq \# X+
M-1$ for all $x \in X$. Then we define
\[ \widehat{X}_M:=\{0,1,2,\ldots, \# X+M-1\}\setminus \{\# X+M-1-x 
\mid x \in X\}.\]
Finally, given two finite subsets $X,Y \subseteq \Z_{\geq 0}$ such that
$\# X=\# Y$ and $X^+=Y^+$, it makes sense to define $X \trianglelefteq Y$.
(Just arrange the entries of $X$ and $Y$ in decreasing order and argue as
in Remark~\ref{rem11}.)

It is well-known that these definitions can be interpreted in terms
of partitions. Let us briefly recall how this works. Let $\lambda$
be a partition of some integer $k \geq 0$. Let $m \geq 0$ be a (large)
integer and write
\[ \lambda=(\lambda_1\geq \lambda_2 \geq \ldots \geq \lambda_m \geq 0).\]
The corresponding $\beta$-set is defined as
\[ B_m(\lambda):=\{\lambda_i+m-i \mid 1 \leq i \leq m\}.\]
Thus, $B_m(\lambda)$ is a finite subset of $\Z_{\geq 0}$ such that
$\# B_m(\lambda)=m$ and $B_m(\lambda)^+=k+\binom{m}{2}$. Conversely,
given a finite subset $X \subseteq \Z_{\geq 0}$, with $m$ elements and
such that $X^+=k+\binom{m}{2}$ where $k \geq 0$, there exists a unique
partition $\lambda$ of $k$ such that $X=B_m(\lambda)$. Thus, every finite
subset $X \subseteq \Z_{\geq 0}$ can be regarded as the $\beta$-set of a
partition of a suitable integer $k \geq 0$. One immediately checks that
\[ \lambda \trianglelefteq \mu \qquad \Leftrightarrow \qquad 
B_m(\lambda) \trianglelefteq B_m(\mu)\]
for all partitions $\lambda,\mu$ of $k$. Furthermore, given a finite
subset $X \subseteq \Z_{\geq 0}$ and an integer $M \geq 0$ as above,
the set $\widehat{X}_M$ has the following interpretation. Write $X=B_m
(\lambda)$ where $m=\# X$ and $\lambda$ is a partition. Then, by
\cite[I.1.7]{Mac}, we have
\[ \widehat{X}_M=\{0,1,\ldots,m+M-1\} \setminus \{m+M-1-x 
\mid x \in X\}=B_M(\overline{\lambda}),\]
where $\overline{\lambda}$ is the conjugate partition.

\begin{lem} \label{strange1} Let $X,Y \subseteq \Z_{\geq 0}$ be
two finite subsets such that $\# X=\# Y$ and $X^+=Y^+$. Let $M \geq 0$
be an integer such that $x\leq \# X+M-1$ for all $x \in X$ and
$y \leq \#Y +M-1$ for all $y \in Y$.  Then we have
\[ X \trianglelefteq Y \qquad \Rightarrow \qquad \widehat{Y}_M
\trianglelefteq \widehat{X}_M.\]
\end{lem}

\begin{proof} Let us write $X=B_m(\lambda)$ and $Y=B_m(\mu)$ where
$m \geq 0$ and $\lambda, \mu$ are partitions of the same integer.
Since $X\trianglelefteq Y$, we have $\lambda \trianglelefteq \mu$. By
\cite[I.1.11]{Mac}, we then also have $\overline{\mu} \trianglelefteq 
\overline{\lambda}$ (where $\overline{\lambda}$ and $\overline{\mu}$
are the conjugate partitions). Furthermore, as explained above, we have
$B_M(\overline{\lambda})=\widehat{X}_M$ and $B_M(\overline{\mu})=
\widehat{Y}_M$. Since $\overline{\mu} \trianglelefteq 
\overline{\lambda}$, we then also have $\widehat{Y}_M 
\trianglelefteq \widehat{X}_M$.
\end{proof}

If $a=2$ and $b'=1$, then the above result immediately implies that
Lemma~\ref{lem11} holds. (Just note that $\overline{Z}$, as defined in
Section~\ref{sec1}, is nothing but $\widehat{Z}_M$ in this case, for
suitable values of $t,M$.) In order to deal with the general case, we
need the following result.

\begin{lem} \label{strange0} Let $\lambda$ and $\mu$ be partitions of
$k$. Let $l \geq 0$ be an integer. Denote by $\lambda'$ the partition
obtained by adding $l$ as a new part to $\lambda$. Similarly, let
$\mu'$ be the partition obtained by adding $l$ as a new part to $\mu$.
(Thus, $\lambda'$ and $\mu'$ are partitions of $k+l$.)
Then we have $\lambda\trianglelefteq \mu$ if and only if
$\lambda' \trianglelefteq \mu'$.
\end{lem}

\begin{proof} We note that the operation of adding $l$ to $\lambda$ is
equivalent to increasing the largest $l$ entries of the conjugate
partition $\overline{\lambda}$ by~$1$; a similar statement also holds
for $\mu$. Now assume that $\lambda\trianglelefteq \mu$. Then we have
$\overline{\mu} \trianglelefteq \overline{\lambda}$ by
\cite[I.1.11]{Mac}. One immediately checks that increasing the largest
$l$ entries of $\overline{\lambda}$ and $\overline{\mu}$ does not effect
the dominance order relation; that is, we also have $\overline{\mu}'
\trianglelefteq \overline{\lambda}'$. But then \cite[I.1.11]{Mac}
implies again that $\lambda'\trianglelefteq \mu'$, as desired. The
argument for the reverse implication is analogous.
\end{proof}

\begin{cor} \label{strange2} Let $X,Y,U \subseteq \Z_{\geq 0}$ be
finite subsets such that $\# X=\# Y$, $X^+=Y^+$ and $U \cap X=U\cap Y=
\varnothing$. Then we have
\[ X \trianglelefteq Y \qquad \Rightarrow \qquad U \cup X 
\trianglelefteq U \cup Y.\]
\end{cor}

\begin{proof} 
The assertion follows by applying  Lemma~\ref{strange0} repeatedly.
\end{proof}

We can now complete the proof of Lemma~\ref{lem11}, as follows. Let
$Z,Z'\in \cM_{a,b;n}^N$ be such that $Z \trianglelefteq Z'$. Let $t$
be a sufficiently large integer such that both $\overline{Z}$ and
$\overline{Z}'$ are defined and in $\cM_{a,b;n}^{t+1-N-r}$; see
Section~\ref{sec1}. Thus, $\overline{Z}$ and $\overline{Z}'$ are obtained
by taking the complements of
\[ \{ta+b' -z \mid z \in Z\} \qquad \mbox{and} \qquad 
\{ta+b'-z' \mid z' \in Z'\}\]
in the set $Z_{a,b'}^t=\{0,a,2a,\ldots,ta,b',a+b',2a+b',\ldots,ta+b'\}$.
Let $U$ be the set of all $i \in \{0,1,2,3,\ldots,ta+b'\}$ such that
$i \not\equiv 0 \bmod a$ and $i \not\equiv b' \bmod a$. Then
\[ \{0,1,2,3,\ldots,ta+b'\}=U \cup Z_{a,b'}^t\qquad \mbox{(disjoint
union)}.\]
(In the special case where $a=2$ and $b'=1$, we have $U=\varnothing$.)
Clearly, we also have
\begin{align*}
\overline{Z}&=\{0,1,2,3,\ldots, ta+b'\}\setminus \bigl(U \cup
\{ta+b'-z \mid z \in Z\}\bigr),\\
\overline{Z}'&=\{0,1,2,3,\ldots, ta+b'\}\setminus \bigl(U \cup
\{ta+b'-z' \mid z \in Z'\}\bigr).
\end{align*}
Now note that $U=\{ta+b'-u \mid u \in U\}$. This yields
\begin{alignat*}{2}
\overline{Z}&=\{0,1,2,3,\ldots, ta+b'\}\setminus \{ta+b'-x \mid x \in U 
\cup Z\}&&=\widehat{(U\cup Z)}_M,\\
\overline{Z}'&=\{0,1,2,3,\ldots, ta+b'\}\setminus \{ta+b'-x' \mid x' \in 
U\cup Z'\}&&=\widehat{(U\cup Z')}_M,
\end{alignat*}
where $M=ta+b'-\# (U \cup Z)+1$.
Now, since $Z \trianglelefteq Z'$, we also have $U \cup Z 
\trianglelefteq U \cup Z'$; see Corollary~\ref{strange2}. Then
Lemma~\ref{strange1} yields that
\[ \overline{Z}'=\widehat{(U \cup Z')}_M \trianglelefteq 
\widehat{(U \cup Z)}_M=\overline{Z},\]
as desired. This completes the proof of Lemma~\ref{lem11}.

\begin{exmp} \label{expstrange} Let $n=3$, $a=3$ and $b=1$. Then
$r=0$, $b'=1$. Consider the bipartitions $(\lambda,\mu)=(11,1)$ and
$(\lambda',\mu')=(21,\varnothing)$. Choosing $N=2$, we have
\[ Z:=Z_{3,1;3}^2(\lambda,\mu)=\{7,6,4,0\} \quad \mbox{and} \quad
Z':=Z_{3,1;3}^2(\lambda',\mu')=\{10,4,3,0\}.\]
We see that $Z \trianglelefteq Z'$, that is, $(11,1) \preceq_{3,1}
(21,\varnothing)$. Choosing $t=3$, we have $ta+b'=10$. Considering the
appropriate complements in $\{0,3,6,9,1,4,7,10\}$ we obtain
\[ \overline{Z}=\{9,7,1,0\} \quad \mbox{and}
\quad \overline{Z}'=\{9,4,3,1\}.\]
Thus, we have $\overline{Z}' \trianglelefteq \overline{Z}$, as required. Now
let us see how this fits with the above argument. We have $U=\{8,5,2\}=
\{10-8,10-5,10-2\}$. This yields
\[  U\cup Z=\{8,7,6,5,4,2,0\} \quad \mbox{and} \quad
 U\cup Z'=\{10,8,5,4,3,2,0\}.\]
Now $\# (U \cup Z)=\# (U\cup Z')=7$ and $M=ta+b'-\# (U \cup Z) +1=4$.
This yields the required equalities $\overline{Z}=\widehat{(U\cup Z)}_4$ and
$\overline{Z}'=\widehat{(U\cup Z')}_4$.
\end{exmp}

%%%%%%%%%%%%%%%%%%%%%%%%%%%%%%%%%%%%%%%%%%%%%%%%%%%%%%%%%%%%%%%%%%%%%%%%%%%
\section{Example: the asymptotic case} \label{secasym}

Throughout this section, we shall assume that  $a,b \in \Z$ are such
that
\begin{center}
\fbox{$\; b >(n-1)a>0.\;$}
\end{center}
This is the ``asymptotic case'' considered by Bonnaf\'e--Iancu \cite{BI2}, 
\cite{BI}. We shall see that, in this case, the pre-order relation 
$\preceq_{a,b}$ admits a more direct and familiar description, without 
reference to the parameters $a$ and $b$. The proof of this description 
would simplify drastically if $b$ were assumed to be very large with
respect to~$a$ (e.g., $b>2n a$). Assuming only that $b>(n-1)a>0$ requires 
a careful analysis of the arrangement of the entries of $Z_{a,b}^N
(\lambda,\mu)$.

Let $r$, $b'$ be as in Section~\ref{sec1}. Then $r \geq n-1$ and 
$0 \leq b'<a$; furthermore, we must have $b'>0$ if $r=n-1$. Given a 
partition $\nu$ of some integer $m$, we denote by $\nu_1 \geq \nu_2 \geq 
\ldots \geq 0$ the parts of $\nu$ and write $|\nu|=m$. 

\begin{lem} \label{lemasy1} Assume that $\lambda=\varnothing$. If
$\mu_1=n$ and $r=n-1$, then the largest entry of $Z_{a,b}^N(\lambda,\mu)$
is $(N+n-1)a$. Otherwise, the largest entry of $Z_{a,b}^N(\lambda,
\mu)$ is $(N+r-1)a+b'$.  
\end{lem}

\begin{proof} The entries of $Z_{a,b}^N(\lambda,\mu)$ which arise from
$\lambda$ are $(N+r-i)a+b'$ for $1\leq i \leq N+r$. The largest of
these entries is $(N+r-1)a+b'$. On the other hand, the largest entry
of $Z_{a,b}^N(\lambda,\mu)$ arising from $\mu$ is $(\mu_1+N-1)a$. This 
immediately yields the assertion for the case where $\mu_1=n$ and 
$r=n-1$. Otherwise, we have $\mu_1\leq r$ and so $N+r-1\geq\mu_1+N-1$.
Hence, in this case, the largest entry of $Z_{a,b}^N(\lambda,\mu)$ is 
$(N+r-1)a+b'$. 
\end{proof}

\begin{rem} \label{lemasy2a} Let $l \geq 1$ be such that 
$\lambda_l+N+r-l \geq \mu_1+N-1$. Then, clearly, the largest $l$ entries 
of $Z_{a,b}^N(\lambda,\mu)$ are $(\lambda_i+N+r-i)a+b'$ for $1 \leq i 
\leq l$. This simple remark will be used frequently in what follows.
\end{rem}

\begin{lem} \label{lemasy2} Assume that $\lambda \neq \varnothing$ and
let $l=|\lambda|\geq 1$; choose $N$ large enough to have $l\leq N+r$. Then
\[\lambda_l+N+r-l\geq \mu_1+N-1\]
and so the largest $l$ entries of $Z_{a,b}^N (\lambda,\mu)$ are
$(\lambda_i+N+r-i)a+b'$ for $1\leq i\leq l$.
\end{lem}

\begin{proof} Since $\mu_1\leq |\mu|=n-|\lambda|=n-l$ and $r \geq n-1$,
we have $\mu_1+N-1\leq N+n-l-1 \leq N+r-l\leq \lambda_l+N+r-l$, as
required. The statement about the largest $l$
entries of $Z_{a,b}^N (\lambda,\mu)$ follows by Remark~\ref{lemasy2a}.
\end{proof}

\begin{defn} \label{defdouble}
The dominance order between bipartitions $(\lambda,\mu)$ and $(\lambda',
\mu')$ of $n$ is defined by
\[(\lambda,\mu) \trianglelefteq (\lambda',\mu') \quad 
\stackrel{\text{def}}{\Leftrightarrow}\quad \left\{\begin{array}{rcll} 
\displaystyle \sum_{1\leq i\leq d} \lambda_i &\leq &\displaystyle 
\sum_{1\leq i\leq d} \lambda_i' &\; \mbox{for all $d$},\\
\displaystyle |\lambda|+ \sum_{1\leq i\leq d} \mu_i &\leq & 
\displaystyle |\lambda'|+ \sum_{1\leq i\leq d} \mu_i' &\; 
\mbox{for all $d$}, \end{array} \right.\]
where $\lambda=(\lambda_1\geq \lambda_2 \geq \ldots \geq 0)$,
$\lambda'= (\lambda_1' \geq \lambda_2' \geq \ldots \geq \ 0)$,
$\mu=(\mu_1\geq \mu_2 \geq \ldots \geq 0)$ and $\mu'=(\mu_1'\geq 
\mu_2'\geq \ldots \geq 0)$. (This appeared in the work of
Dipper--James--Murphy \cite{DJM3}.)
\end{defn}

\begin{rem} \label{remdouble}  Let us  define the dominance
order also for partitions of possibly different numbers (as, for example,
in the section on ``raising operators'' in \cite[I.1]{Mac}). Then we
have:
\[(\lambda,\mu) \trianglelefteq (\lambda',\mu') \qquad \Leftrightarrow 
\qquad \lambda \trianglelefteq \lambda' \quad \mbox{and} \quad 
\overline{\mu}' \trianglelefteq \overline{\mu}.\]
This easily follows by using a slight variation of the argument in
\cite[I.1.11]{Mac}.
\end{rem}

\begin{rem} \label{remadjdouble} Let $(\lambda, \mu)$ and $(\lambda',\mu')$
be bipartitions of $n$ such that $(\lambda,\mu) \trianglelefteq (\lambda',
\mu')$ and $(\lambda,\mu) \neq (\lambda',\mu')$. Assume that $(\lambda,\mu)$
and $(\lambda',\mu')$ are adjacent with respect to $\trianglelefteq$ (that
is, if $(\lambda,\mu)\trianglelefteq (\lambda'',\mu'') \trianglelefteq 
(\lambda',\mu')$ for some bipartition $(\lambda'',\mu'')$, then
$(\lambda,\mu)=(\lambda'',\mu'')$ or $(\lambda'',\mu'')=(\lambda',\mu')$).
Then we are in one of the following three cases.
\begin{itemize}
\item[(a)] $\mu=\mu'$ and there exist indices $i<j$ such that $\lambda'$
is obtained from $\lambda$ by increasing $\lambda_i$ by $1$ and decreasing
$\lambda_j$ by $1$. (In particular, $\lambda,\lambda'$ are adjacent with
respect to $\trianglelefteq$.)
\item[(b)] $\lambda=\lambda'$ and there exist indices $i<j$ such
that $\mu'$ is obtained from $\mu$ by increasing $\mu_i$ by $1$ and
decreasing $\mu_j$ by $1$. (In particular, $\mu,\mu'$ are adjacent with
respect to $\trianglelefteq$.)
\item[(c)] $|\lambda|<|\lambda'|$ and there exist indices $i,j$ such
that $\lambda'$ is obtained from $\lambda$ by increasing $\lambda_i$
by $1$ and $\mu'$ is obtained from $\mu$ by decreasing $\mu_j$ by $1$.
\end{itemize}
This follows easily by an argument similar to \cite[I.1.16]{Mac}.
\end{rem}

\begin{prop} \label{lemasy3} Let $(\lambda,\mu)$ and $(\lambda',\mu')$
be bipartitions of $n$. The following holds 
\[ (\lambda,\mu) \preceq_{a,b}
(\lambda',\mu')\quad\Rightarrow\quad \lambda\trianglelefteq \lambda'\, .\]
\end{prop}

\begin{proof} If $\lambda=\varnothing$, then this is clear. Now 
assume that $\lambda'=\varnothing$. We show that then we also have 
$\lambda=\varnothing$. Assume that this is not the case. Then, by 
Lemma~\ref{lemasy2}, the largest entry of $Z_{a,b}^N(\lambda, \mu)$ 
is $(\lambda_1+N+r-1)a+ b'$. The condition $Z_{a,b}^N(\lambda, \mu)
\trianglelefteq Z_{a,b}^N(\lambda',\mu')$ implies, in particular, 
that the largest entry of $Z_{a,b}^N(\lambda,\mu)$ is less than or 
equal to the largest entry of $Z_{a,b}^N(\lambda', \mu')$. Using 
Lemma~\ref{lemasy1}, we distinguish two cases. If $\mu_1'=n$ and 
$r=n-1$, we must have $(\lambda_1+N+n-2)a+b' \leq (N+n-1)a$.
Since then also $b'>0$, this forces that $\lambda_1=0$, a contradiction.
Otherwise, we have $(\lambda_1+N+ r-1)a+b' \leq (N+r-1)a+b'$ 
which yields again that $\lambda_1=0$, a contradiction. Thus, if 
$\lambda'= \varnothing$, then $\lambda=\varnothing$, as required.

We can now assume that $\lambda\neq \varnothing$ and $\lambda'\neq 
\varnothing$. Let 
\[ k=\max\{i \geq 1\mid \lambda_i>0\} \qquad \mbox{and} \qquad 
k'=\max\{i \geq 1\mid \lambda'_i>0\}.\]
Clearly, we have $k \leq |\lambda|$ and $k'\leq |\lambda'|$. So, by 
Lemma~\ref{lemasy2}, we have
\[ \lambda_k+N+r-k \geq \mu_1+N-1 \qquad \mbox{and} \qquad 
\lambda_{k'}'+N+r-k' \geq \mu_1'+N-1\, ;\]
also the largest $k$ entries of $Z_{a,b}^N(\lambda,\mu)$
are $(\lambda_i+N+r-i)a+b$ for $1\leq i\leq k$, while the
largest $k'$ entries of $Z_{a,b}^N(\lambda',\mu')$   are 
$(\lambda'_i+N+r-i)a+b'$ for $1\leq i\leq k'$.

Now we distinguish two cases. First case: we also have 
$\lambda_k'+N+r-k \geq \mu_1'+ N-1$. Let $d$ be such that $1\leq d\leq k$;
then, by Remark~\ref{lemasy2a},
the sum of the largest $d$ entries of $Z_{a,b}^N(\lambda,\mu)$ is 
\[ \sum_{1\leq i \leq d} \bigl((\lambda_i+N+r-i)a+b'\bigr),\]
while the sum of the largest $d$ entries of $Z_{a,b}^N(\lambda', 
\mu')$ is 
\[\sum_{1 \leq i \leq d} \bigl((\lambda_i'+N+r-i)a+b'\bigr).\]
Now the condition $Z_{a,b}^N(\lambda, \mu)\trianglelefteq Z_{a,b}^N
(\lambda',\mu')$ implies, in particular, that the first sum is less than 
or equal to the second sum. Hence, we obtain
\[ \sum_{1\leq i \leq d} \lambda_i \leq
\sum_{1\leq i \leq d} \lambda_i',\quad\text{for all}\; d\in\{ 1,\ldots ,k\}\]
and so $\lambda\trianglelefteq \lambda'$, as desired. 

It remains to consider the second case: where  we have
\[ \lambda_k'+N+r-k <\mu_1'+N-1.\]
Note that, since $\lambda_{k'}'+N+r-k' \geq \mu_1'+N-1$, this forces
$k>k'$. So there exists an index $l \in \{k'+1,\ldots,k\}$ such that 
\begin{align*}
\lambda_{i}'+N+r-i&\geq \mu_1'+N-1 \qquad \mbox{for $i=k',k'+1,
\ldots,l-1$},\\ 
\lambda_{l}'+N+r-l&<\mu_1'+N-1.
\end{align*}
Consequently, the $l$ largest entries  of $Z_{a,b}^N(\lambda',\mu')$ are
\[ (\lambda_i'+N+r-i)a+b' \quad (1\leq i \leq l-1) \qquad \mbox{and} \qquad
(\mu_1'+N-1)a.\]
Since $l\leq k$, the $l$ largest entries of $Z_{a,b}^N(\lambda,\mu)$ are
$(\lambda_i+N+r-i)a+b'$ for $1\leq i \leq l$. Taking the sums of these
entries, the condition $Z_{a,b}^N (\lambda, \mu) \trianglelefteq Z_{a,b}^N
(\lambda', \mu')$ implies, in particular, that 
\[\sum_{1\leq i \leq l} \bigl((\lambda_i+
N+r-i)a+b'\bigr)
\leq (\mu_1'+N-1)a+\sum_{1 \leq i \leq l-1} 
\bigl((\lambda_i'+N+r-i)a+b'\bigr).\]
Hence, we obtain 
\[ (\lambda_l+r-l)a+b'+\Bigl(\sum_{1 \leq i \leq l-1} \lambda_i\Bigr)
a \leq (\mu_1'-1)a+\Bigl(\sum_{1 \leq i \leq l-1} \lambda_i'\Bigr)a\]
and so 
\[ ra+b'-la+\Bigl(\sum_{1 \leq i \leq l} \lambda_i\Bigr)a\leq 
\mu_1'a-a+|\lambda'|a.\]
Finally, since, $\mu_1'\leq |\mu'|$, we deduce that 
\[ra+b' \leq  \mu_1'a-a +|\lambda'|a+\Bigl(l-\sum_{1 \leq i \leq l} 
\lambda_i\Bigr)a\leq (n-1)a+\Bigl(l-\sum_{1 \leq i \leq l} 
\lambda_i\Bigr)a.\]
Since $l\leq k$, we have $\lambda_i\geq 1$ for $1\leq i \leq l$.
Hence, the right hand side of the above inequality is less than or
equal to $(n-1)a$. We conclude that $b=ra+b'\leq (n-1)a$, a contradiction. 
\end{proof}

%Assume that $r=n-1$ and $b'=0$. Then last inequality implies
%that $\mu_1'=|\mu'|$ and $\lambda_i=1$ for $1\leq i \leq l$.
%Consequently, all parts of $\lambda$ are $1$ and $\mu'$ has only
%one non-zero part. Now, if $(\lambda,\mu)$ and $(\lambda',\mu')$
%belong to the same combinatorial family, then can exchange the
%roles of $(\lambda,\mu)$ and $(\lambda',\mu')$ in the above argument.
%Hence, all parts of $\lambda'$ must be $1$ and $\mu$ has only
%one non-zero part. This yields sub-asymptotic case.

\begin{cor} \label{propasy} Recall that $b>(n-1)a>0$. Let $(\lambda,
\mu)$ and $(\lambda',\mu')$ be bipartitions of $n$. Then we have
\[(\lambda,\mu)\preceq_{a,b} (\lambda',\mu') \qquad \Leftrightarrow 
\qquad (\lambda,\mu) \trianglelefteq (\lambda',\mu').\]
In particular, $\preceq_{a,b}$ is not just a pre-order but a partial
order.
\end{cor}

\begin{proof} Assume first that $(\lambda,\mu)\preceq_{a,b} (\lambda',
\mu')$. Then, by Proposition~\ref{cor11}, we know that we also have
$(\overline{\mu}',\overline{\lambda}')\preceq_{a,b} 
(\overline{\mu},\overline{\lambda})$. Now Proposition~\ref{lemasy3}
implies that $\lambda\trianglelefteq\lambda'$ and 
$\overline{\mu}' \trianglelefteq \overline{\mu}$, as desired.

Conversely, assume that $(\lambda,\mu) \trianglelefteq (\lambda',\mu')$.
In order to show that $(\lambda,\mu) \preceq_{a,b} (\lambda',\mu')$, 
we may assume without loss of generality that $(\lambda,\mu) \neq 
(\lambda',\mu')$ and that $(\lambda,\mu)$ and $(\lambda',\mu')$ are adjacent 
with respect to $\trianglelefteq$. Thus, we are in one of the three
cases in Remark~\ref{remadjdouble}. In each of these cases, one can
check directly that $(\lambda,\mu) \preceq_{a,b} (\lambda',\mu')$. 
We omit the details as this will also follow by combining 
Theorem~\ref{mainbn} with the implication ($*$) in 
Example~\ref{expasym2} below. 
\end{proof}

%\begin{lem} \label{equneq} Assume that $b'>0$. Let $(\lambda,\mu)$
%and $(\lambda',\mu')$ be bipartitions of $n$. Then 
%\[ (\lambda,\mu) \preceq_{a,b} (\lambda',\mu') \qquad \Leftrightarrow\qquad
%(\lambda,\mu) \preceq_{2,2r+1} (\lambda',\mu').\] 
%\end{lem}
%
%\begin{proof} ??? 
%\end{proof}
%\marginpar{TODO or leave for later}

%%%%%%%%%%%%%%%%%%%%%%%%%%%%%%%%%%%%%%%%%%%%%%%%%%%%%%%%%%%%%%%%%%%%%%%%%%%
\section{The pre-order relation $\preceq_L$ in type $B_n$} \label{sec3}

Throughout this section, let $n \geq 1$ and $W_n$ be a Coxeter group 
of type $B_n$, with generators and diagram given by 
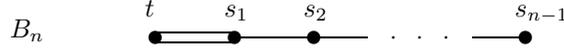
\begin{center}
\begin{picture}(250,20)
\put( 10, 5){$B_n$}
\put( 61,13){$t$}
\put( 91,13){$s_1$}
\put(121,13){$s_2$}
\put(201,13){$s_{n{-}1}$}
\put( 65, 5){\circle*{5}}
\put( 95, 5){\circle*{5}}
\put(125, 5){\circle*{5}}
\put( 65, 7){\line(1,0){30}}
\put( 65, 3){\line(1,0){30}}
\put( 95, 5){\line(1,0){30}}
\put(125, 5){\line(1,0){20}}
\put(155,5){\circle*{1}}
\put(165,5){\circle*{1}}
\put(175,5){\circle*{1}}
\put(185, 5){\line(1,0){20}}
\put(205, 5){\circle*{5}}
\end{picture}
\end{center}
As in Example~\ref{bnsetup}, we write $\Irr(W_n)=\{E^{(\lambda, \mu)} 
\mid (\lambda,\mu) \mbox{ is a bipartition of } n \}$. A weight function 
$L \colon W_n \rightarrow \Z$ is specified by the two parameters 
$b:=L(t)$ and $a:=L(s_i)$ for $1\leq i \leq n-1$. From now until 
Example~\ref{bnanull} (at the very end of this section), we shall assume that 
\begin{center}
\fbox{$\;b \geq 0$ and $a>0.\;$}
\end{center}
(It is convenient to allow the possibility that $b=0$ because this is 
related to groups of type $D_n$; see Example~\ref{expdn} in the following 
section.) The main result of this section is Theorem~\ref{mainbn}, which 
provides a combinatorial description of the pre-order relation $\preceq_L$ 
on $\Irr(W_n)$ in terms of the pre-orders on bipartitions considered in
Section~\ref{sec1}.

We begin by collecting some known results. The effect of tensoring with 
the sign representation is given as follows.

\begin{lem} \label{bnsign} Let $(\lambda,\mu)$ be a bipartition of $n$. 
Then $E^{(\lambda,\mu)} \otimes \sgn=E^{(\overline{\mu},
\overline{\lambda})}$.
\end{lem}

\begin{proof} This can be found, for example, in \cite[5.5.6]{gepf}.
\end{proof}

\begin{prop}[Lusztig \protect{\cite[22.14]{Lusztig03}}] \label{bnafunc}
Let $(\lambda,\mu)$ be a bipartition of $n$. Then
\[ \ba_{E^{(\lambda,\mu)}}=\ba_{a,b}(\lambda,\mu); \qquad \mbox{see 
Definition~\ref{def2}}.\]
\end{prop}

(Note that Lusztig assumes that $b>0$ but this formula also works for
$b=0$.)

\begin{rem} \label{remomg1} By Remark~\ref{defafun1} and 
Lemma~\ref{bnsign}, we have
\[ \omega_L(E^{(\lambda,\mu)})=\ba_{a,b}(\overline{\mu},
\overline{\lambda})-\ba_{a,b}(\lambda,\mu).\]
Using \cite[Lemmas~6.2.6 and 6.2.8]{gepf}, we obtain the 
following more direct formula:
\[\omega_L(E^{(\lambda,\mu)})=\bigl(|\lambda|-|\mu|\bigr)\,b+
2\bigl(n(\overline{\lambda})-n(\lambda)+n(\overline{\mu})-
n(\mu)\bigr)\,a,\]
where $n(\nu)$ (for any partition $\nu$) is defined as in 
Example~\ref{ordA}.
\end{rem}

Using the combinatorics in Section~\ref{sec1}, Lusztig has determined
the families of $\Irr(W_n)$. It turns out that these are precisely given 
by the ``combinatorial families'' in Definition~\ref{def1}.

\begin{prop}[Lusztig \protect{\cite[23.1]{Lusztig03}}] \label{bnfam}
Let $(\lambda,\mu)$ and $(\lambda',\mu')$ be bipartitions of $n$. Then:
\[ E^{(\lambda,\mu)}, E^{(\lambda',\mu')} \mbox{ belong to the same 
family} \quad \Leftrightarrow \quad Z_{a,b}^N(\lambda,\mu)=
Z_{a,b}^N(\lambda',\mu'),\]
where $N \geq 0$ is a sufficiently large integer 
(see Definition~\ref{def0}). 
\end{prop}

Let us now turn to the description of the pre-order relation $\preceq_L$. 
To state the following result, we recall that the maximal parabolic
subgroups of $W_n$ are of the form $W_k \times H_l$ where $n=k+l$ 
($k \geq 0$, $l\geq 1$). Here, $W_k$ is of type $B_k$ (generated by $t,s_1,
\ldots,s_{k-1}$) and $H_l$ is of type $A_{l-1}$ (generated by $s_{k+1},
s_{k+2}, \ldots,s_{n-1}$). It is understood that $W_0=H_1=\{1\}$. Let 
$\sgn_l$ denote the sign representation of $H_l$.

\begin{lem}[Cf. Spaltenstein \protect{\cite[\S 3]{Spalt}}] \label{ordB1} 
Let $E,E'\in \Irr(W_n)$. Then $E \preceq_L E'$ if and only if there exists 
a sequence $E=E_0,E_1,\ldots,E_m=E'$ in $\Irr(W_n)$ such that, for each 
$i \in \{1,2,\ldots,m\}$, the following condition is satisfied: \\There 
exists a decomposition $n=k_i+l_i$ ($k_i\geq 0$, $l_i\geq 1$) and $M_i,
M_i'\in\Irr(W_{k_i})$ where $M_i\preceq_L M_i'$ within $\Irr(W_{k_i})$, 
such that either 
\begin{align*}
M_i \boxtimes \sgn_{l_i} \uparrow E_{i-1} \qquad &\mbox{and}
\qquad M_i'\boxtimes \sgn_{l_i} \rightsquigarrow_L E_i\\\intertext{or}
M_i \boxtimes \sgn_{l_i}\uparrow E_i \otimes \sgn \qquad &
\mbox{and} \qquad M_i' \boxtimes \sgn_{l_i} \rightsquigarrow_L 
E_{i-1}\otimes \sgn.
\end{align*}
(Here, $M_i \boxtimes \sgn_{l_i}$ and $M_i'\boxtimes \sgn_{l_i}$ are 
representations of $W_{k_i} \times H_{l_i} \subseteq W_n$.)
\end{lem}

\begin{proof} The ``if'' part is clear by the definition of $\preceq_L$.
To prove the ``only if'' part, it is sufficient to consider an elementary
step in Definition~\ref{mydef}. That is, we can assume that there is a 
subset $I \subsetneqq S$ and $M_1,M_1'\in \Irr(W_I)$, where $M_1 \preceq_L 
M_1'$ within $\Irr(W_I)$, such that one of the following two conditions 
holds.
\begin{itemize}
\item $M_1\uparrow E$ and $M_1' \rightsquigarrow_L E'$.
\item $M_1 \uparrow E'\otimes \sgn$ and $M_1' \rightsquigarrow_L 
E\otimes \sgn$.
\end{itemize}
By Remark~\ref{ordmax}, we can further assume that $W_I$ is a maximal
parabolic subgroup of $W_n$, that is, we have $W_I=W_k\times H_l$ where
$k \geq 0$, $l\geq 1$. Since $H_l \cong \fS_l$, we can write
\[ M_1=\tilde{M}_1 \boxtimes E^{\lambda} \qquad \mbox{and}\qquad M_1'=
\tilde{M}_1'\boxtimes E^{\mu}\]
where $\tilde{M}_1,\tilde{M}_1'\in \Irr(W_k)$ and $\lambda$, $\mu$ are 
partitions of $l$. By Remark~\ref{ordprod} and Example~\ref{ordA}, we have
\[ \tilde{M}_1 \preceq_L \tilde{M}_1' \quad \mbox{(within $\Irr(W_k)$)} 
\qquad \mbox{and} \qquad \lambda \trianglelefteq \mu.\]
Let $H_{\overline{\mu}} \subseteq H_l$ be the parabolic subgroup 
corresponding to the Young subgroup $\fS_{\overline{\mu}} \subseteq\fS_l$. 
By Example~\ref{ordA} (see the proof of the implication ``(b) 
$\Rightarrow$ (c)''), we have $\sgn_{\overline{\mu}} \uparrow E^\lambda$ 
and $\sgn_{\overline{\mu}} \rightsquigarrow_L E^\mu$ where 
$\sgn_{\overline{\mu}}$ is the sign representation of $H_{\overline{\mu}}$.
Hence, by Remark~\ref{ordprod} and the transitivity of induction, one 
of the following two conditions holds.
\begin{itemize}
\item $\tilde{M}_1\boxtimes \sgn_{\overline{\mu}} \uparrow E$ and 
$\tilde{M}_1'\boxtimes \sgn_{\overline{\mu}} \rightsquigarrow_L E'$.
\item $\tilde{M}_1\boxtimes \sgn_{\overline{\mu}}\uparrow E'\otimes 
\sgn$ and $\tilde{M}_1' \boxtimes \sgn_{\overline{\mu}} 
\rightsquigarrow_L E\otimes \sgn$.
\end{itemize}
Thus, we have replaced the maximal parabolic subgroup $W_I=W_k \times 
H_l$ (that we started with) by the parabolic subgroup $W_k \times 
H_{\overline{\mu}}$, where we consider the sign representation on the 
$H_{\overline{\mu}}$-factor. We will now embed $W_k \times 
H_{\overline{\mu}}$ into a different maximal parabolic subgroup such 
that the required conditions will be satisfied. 

For this purpose, let $\overline{\mu}=(\overline{\mu}_1\geq
\overline{\mu}_2\geq\ldots\geq \overline{\mu}_d \geq 1)$ be the non-zero 
parts of $\overline{\mu}$. Correspondingly, we have a direct product 
decomposition 
\[ H_{\overline{\mu}}=H_{\overline{\mu}_1}\times H_{\overline{\mu}_2}
\times \cdots \times H_{\overline{\mu}_d}.\]
Grouping the first $d-1$ factors together, we obtain 
\[ W_k \times H_{\overline{\mu}} \subseteq (W_k \times H_{l_1}) \times 
H_{l'} \subseteq W_{k'} \times H_{l'}\]
where $l_1:=\overline{\mu}_1+\overline{\mu}_2+ \ldots + 
\overline{\mu}_{d-1}$, $k':=k+l_1$ and $l':=\overline{\mu}_d$. Using 
Remarks~\ref{ordprod} and \ref{ordmax}, we conclude that there exist 
$M,M' \in \Irr(W_{k'})$ such that 
\begin{equation*}
\tilde{M}_1 \boxtimes \sgn_{\overline{\mu}} \uparrow M \boxtimes 
\sgn_{l'}, \qquad \tilde{M}_1' \boxtimes \sgn_{\overline{\mu}} 
\rightsquigarrow_L M'\boxtimes \sgn_{l'}\tag{$*$}
\end{equation*} 
and one of the following two conditions holds.
\begin{itemize}
\item $M\boxtimes \sgn_{l'} \uparrow E$ and $M'\boxtimes 
\sgn_{l'} \rightsquigarrow_L E'$.
\item $M\boxtimes \sgn_{l'} \uparrow E'\otimes \sgn$ and 
$M' \boxtimes \sgn_{l'} \rightsquigarrow_L E\otimes\sgn$.
\end{itemize}
Finally, we use Remark~\ref{ordprod} to conclude that 
$M \preceq_L M'$ within $\Irr(W_{k'})$. Indeed, since $\tilde{M}_1
\preceq_L \tilde{M}'_1$ within $\Irr(W_k)$, we have $\tilde{M}_1
\boxtimes \sgn_{\overline{\mu}} \preceq_L \tilde{M}_1' \boxtimes 
\sgn_{\overline{\mu}}$ within $\Irr(W_k \times H_{\overline{\mu}})$. 
Then ($*$) implies that $M \boxtimes \sgn_{l'} \preceq_L M' \boxtimes 
\sgn_{l'}$ within $\Irr(W_{k'} \times H_{l'})$ and, hence, $M \preceq_L 
M'$ within $\Irr(W_{k'})$, as required.
\end{proof}

In order to proceed, we need some more precise information about the
induction of representations from $W_k \times H_l$ to $W_n$. The basic
tool is the following rule:

\begin{lem}[Pieri's Rule for $W_n$] \label{pieri} Let $n=k+l$ where 
$k \geq 0$, $l \geq 1$. Let $(\alpha,\beta)$ be a bipartition of $k$ 
and $(\lambda,\mu)$ be a bipartition of $n$.  Then we have
\[ E^{(\alpha,\beta)} \otimes \sgn_l \uparrow E^{(\lambda,\mu)}\]
 if and only if $(\lambda,\mu)$ can be obtained by increasing $l$ 
parts of $(\alpha,\beta)$ by $1$.
\end{lem}

\begin{proof} This can be reduced to a statement about representations 
of the symmetric group, where it corresponds to the classical ``Pieri 
Rule'' for symmetric functions. See \cite[6.1.9]{gepf} and the proof 
of \cite[6.4.7]{gepf} for details.
\end{proof}

In order to be able to apply this rule in our context, we need to 
interpret it in terms of multisets. So let $n=k+l$, $(\alpha,\beta)$, 
$(\lambda,\mu)$ be as above and assume that $E^{(\alpha,\beta)} \otimes 
\sgn_l \uparrow E^{(\lambda,\mu)}$.
Let $N\geq 0$ be a sufficiently large integer and consider the multisets 
$Z_{a,b}^N(\alpha,\beta)$ and $Z_{a,b}^N(\lambda, \mu)$. Let 
\[Z_{a,b}^N(\alpha,\beta)=\{u_1,\ldots,u_{2N+r}\} \qquad \mbox{where}\qquad
u_1 \geq u_2 \geq \ldots \geq u_{2N+r}.\]
Then, by Lemma~\ref{pieri}, we have 
\[Z_{a,b}^N(\lambda,\mu)=\{u_i+\delta_i a\mid 1\leq i \leq 2N+r\}
\qquad \mbox{where} \qquad \delta_i\in\{0,1\};\]
furthermore, the number of $i \in \{1,2,\ldots,2N+r\}$ such that 
$\delta_i=1$ equals~$l$. Note, however, that the entries $u_i+\delta_i a$ 
are not necessarily arranged in decreasing order! We need to know 
precisely to what extent this can fail. 

We define a sequence $(z_i)_{1\leq i \leq 2N+r}$ as follows. For all $i$
such that $\delta_i=0$, $\delta_{i+1}=1$ and $u_{i+1}+a>u_i$, we set $z_i
:=u_{i+1}+a$ and $z_{i+1}:= u_i$. For all the remaining $i$, we set $z_i:=
u_i+\delta_i a$. Thus, we have $Z_{a,b}^N(\lambda,\mu)=\{z_1,\ldots,
z_{2N+r}\}$.

\begin{exmp} \label{expmult2} Let $k=14$ and $(\alpha,\beta)=(4311,32)$. 

(a) Let $l=3$ and $(\lambda,\mu)=(4321,421)$. Assume that $a=b=1$. Then 
$r=1$, $b'=0$ and we can take $N=3$. As in Example~\ref{expmult1}, we obtain 
\begin{align*}
 Z_{1,1}^3(4311,32)&=\{7,\hat{5},5,3,\hat{2},1,\hat{0}\},\\
 Z_{1,1}^3(4321,421)&=\{7,6,5,3,3,1,1\},
\end{align*}
where the hat indicates that $\delta_i=1$. Note that, since $u_2=u_3=5$,
the sequence $(\delta_i)$ is not uniquely determined: we could take either
$\delta_2=0$, $\delta_3=1$ or $\delta_2=1$, $\delta_3=0$. But if we choose 
the second possibility, then $z_i=u_i+\delta_i$ for all $i$, and the 
sequence $(z_1,\ldots,z_7)$ is in decreasing order. This will always be 
the case if $b'=0$, assuming that whenever we have an equality $u_i=
u_{i+1}$, then $\delta_i\geq \delta_{i+1}$. 

(b) Let $l=4$ and $(\lambda,\mu)=(4421,331)$. Assume that $a=2$, $b=1$.
Then $r=0$, $b'=1$ and we can take $N=4$. As in Example~\ref{expmult1}, we
obtain
\begin{align*}
 Z_{2,1}^4(4311,32)&=\{15,12,\hat{11},\;\hat{8},\hat{5},3,\hat{2},0\},\\
 Z_{2,1}^4(4421,331)&=\{15,12,13,10,7,3,4,0\},\end{align*}
where the hat indicates that $\delta_i=1$.  Hence, in this case, we see
that some reordering is required. The only critical indices $i$ such that
$\delta_i=0$, $\delta_{i+1}=1$ and $u_{i+1}+2>u_i$ are $i=2$ and $i=6$.
Thus, we set $z_2:=u_3+2=13$, $z_3:=u_2=12$, $z_6:=u_7+2=4$, $z_7:=u_6=3$ 
and $z_i:=u_i+\delta_i$ for all the remaining~$i$. Then $(z_1,\ldots,z_8)=
(15,13,12,10,7,4,3,0)$ is in decreasing order.
\end{exmp}

The observations in this example are true in general:

\begin{lem} \label{pieri1} In the above setting, we have $z_1\geq z_2
\geq \ldots \geq z_{2N+r}$ and 
\[ \sum_{1\leq i \leq d} z_i\leq \min\{d,l\}\, a+ \sum_{1\leq i \leq d} u_i 
\qquad \mbox{for $1\leq d \leq 2N+r$}.\]
\end{lem}

\begin{proof} Assume that $1\leq i \leq 2N+r-1$. We want to show that $z_i\geq
z_{i+1}$.
From the construction, we see that we have the following possibilities:
\[ z_i \in \{u_i,u_{i-1},u_i+a,u_{i+1}+a\}\quad \mbox{and}\quad
 z_{i+1} \in \{u_{i+1},u_i,u_{i+1}+a,u_{i+2}+a\},\]
where the combination $(z_i,z_{i+1})=(u_i,u_{i+1}+a)$ only occurs if
$\delta_i=0$, $\delta_{i+1}=1$, $u_i \geq u_{i+1}+a$ and the combination 
$(z_i,z_{i+1})=(u_{i+1}+a,u_i)$ only occurs if $\delta_i=0$, 
$\delta_{i+1}=1$,
$u_{i+1}+a>u_i$. Further note that, since $Z_{a,b}^N(\alpha,\beta)$ 
satisfies the conditions (M1)--(M3) in Section~\ref{sec1}, we have $u_i -
u_{i+2} \geq a$ for all $i$. Hence, for all valid combinations of pairs
$(z_i,z_{i+1})$ as above, we have $z_i \geq z_{i+1}$, as claimed.

Now consider the sum $\sum_{1\leq i \leq d} z_i$ and note that
if $d\geq l$, then the desired inequality holds. Assume now that 
$d<l$ and let $1\leq i_1<i_2<
\ldots <i_h\leq 2N+r-1$ be the critical indices $i$ such that 
$\delta_i=0$, $\delta_{i+1}=1$ and $u_{i+1}+a>u_i$. Note that $i_{j+1}-i_j
\geq 2$ for all $j$. Now we have two cases: 

If $d\not\in\{i_1,i_2,\ldots,i_h\}$, then $(z_1, \ldots,z_d)$ will 
be equal, up to possibly interchanging consecutive indices, to 
$(u_1+\delta_1a, \ldots,u_d+\delta_da)$. Hence, if we sum over these
two sequences, we get the same result and so
\[ \sum_{1\leq i \leq d} z_i=\Bigl(\sum_{1\leq i\leq d} \delta_i\Bigr)a+
\sum_{1\leq i \leq d} u_i\leq \min\{d,l\}\,a +\sum_{1\leq i \leq d}u_i.\]
On the other hand, if $d=i_j$ for some $j$, then $d-1\not\in\{i_1,i_2,
\ldots,i_h\}$ and so the previous case applies to the sum $\sum_{1\leq i 
\leq d-1} z_i$. Hence, we have
\[ \sum_{1\leq i \leq d} z_i=z_d+\sum_{1\leq i \leq d-1} z_i
\leq z_d+\min\{d-1,l\}\,a+\sum_{1\leq i \leq d-1} u_i.\]
Now $z_d=u_{d+1}+a\leq u_d+a$ and so we obtain again the desired estimation. 
\end{proof}

%Now we distinguish four cases.
%\begin{itemize}
%\item[(1)] Assume that $\delta_i=0$ and $\delta_{i+1}=1$. If $u_{i+1}+a>
%u_i$, then $z_i=u_{i+1}+a$, $z_{i+1}=u_i$ by the above construction and 
%so $z_i\geq z_{i+1}$. If $u_{i+1}+a \leq u_i$, then $z_i=u_i$, $z_{i+1}=
%u_{i+1}+a$ and so $z_i \geq z_{i+1}$.
%\item[(2)] Assume that $\delta_i=\delta_{i+1}=0$. Then $z_i=u_i$.
%Furthermore, $z_{i+1}=u_{i+1}$ or $z_{i+1}=u_{i+2}+a$, where the latter
%possibility arises if $i+1,i+2$ satisfy the condition in (1). Hence,
%using ($*$), we see that $z_i \geq z_{i+1}$.
%\item[(3)] Assume that $\delta_i=\delta_{i+1}=1$. Then $z_{i+1}=u_{i+1}$.
%Furthermore, $z_i=u_i$ or $z_i=u_{i-1}$, where the latter possibility 
%arises if $i-1,i$ satisfy the condition in (1). Hence, we certainly have
%$z_i \geq z_{i+1}$.
%\item[(4)] Assume that $\delta_i=1$ and $\delta_{i+1}=0$. Then
%$z_i=u_i$ or $z_i=u_{i-1}$, and $z_{i+1}=u_{i+1}$ or $z_{i+1}=u_{i+2}+a$.
%In all cases, $z_i \geq z_{i+1}$ as required.
%\end{itemize}

\begin{lem} \label{pieri2} Let $n=k+l$ where $k \geq 0$, $l \geq 1$. 
Let $(\alpha,\beta)$ be a bipartition of $k$ and $(\lambda,\mu)$ be a 
bipartition of $n$. Let $N\geq 0$ be a sufficiently large integer and 
consider the multisets $Z_{a,b}^N(\alpha,\beta)$ and $Z_{a,b}^N(\lambda,
\mu)$. Then the following implication holds:
\[ E^{(\alpha,\beta)} \otimes \sgn_l \uparrow E^{(\lambda,\mu)} \qquad
\Rightarrow \qquad Z_{a,b}^N(\lambda,\mu) \trianglelefteq \hat{Z}_{a,b}^N
(\alpha,\beta),\]
where $\hat{Z}_{a,b}^N(\alpha,\beta)\in \cM_{a,b;n}^N$ is the multiset
obtained by increasing the largest $l$ entries of $Z_{a,b}^N(\alpha,\beta)$ 
by~$a$.
\end{lem}

\begin{proof} Write $Z_{a,b}^N(\alpha,\beta)=\{u_1,\ldots,u_{2N+r}\}$ 
and $Z_{a,b}^N(\lambda,\mu)=\{z_1,\ldots,z_{2N+r}\}$ as above, where the 
entries in both multisets are in decreasing order. We have 
\[\hat{Z}_{a,b}^N(\alpha,\beta)=\{\hat{u}_1,\ldots,\hat{u}_{2N+r}\}
\quad \mbox{where} \quad \hat{u}_i=\left\{\begin{array}{cl} u_i+a & 
\mbox{ for $1\leq i\leq l$},\\u_i &\mbox{ for $i>l$}.\end{array}\right.\]
Hence, using the estimation in Lemma~\ref{pieri1}, we obtain
\[ \sum_{1\leq i \leq d} z_i \leq \min\{d,l\}a+\sum_{1\leq i \leq d} 
u_i= \sum_{1\leq i \leq d} \hat{u}_i\]
for $1\leq d \leq 2N+r$. This means that $Z_{a,b}^N(\lambda,\mu) 
\trianglelefteq \hat{Z}_{a,b}^N(\alpha,\beta)$, as required.
\end{proof}

\begin{lem}[Lusztig \protect{\cite[22.17]{Lusztig03}}] \label{pieri3} 
In the setting of Lemma~\ref{pieri2}, we have 
\[ E^{(\alpha,\beta)} \otimes \sgn_l \rightsquigarrow_L 
E^{(\lambda,\mu)}\qquad \Rightarrow \qquad Z_{a,b}^N(\lambda,\mu)=
\hat{Z}_{a,b}^N(\alpha,\beta).\]
\end{lem}

\begin{proof} Let us write $\hat{Z}_{a,b}^N(\alpha,\beta)=\{\hat{u}_1,
\ldots, \hat{u}_{2N+r}\}$ as in the above proof. Then 
\begin{align*}
\sum_{1\leq i\leq 2N+r} (i-1) \hat{u}_i&=\sum_{1\leq i\leq l}(i-1)a+
\sum_{1\leq i \leq 2N+r} (i-1)u_i\\ &= \binom{l}{2}a+
\sum_{1\leq i \leq 2N+r} (i-1)u_i.
\end{align*}
Thus, we also have 
\[ \ba_{a,b}\bigl(\hat{Z}_{a,b}^N(\alpha,\beta)\bigr)
=\binom{l}{2}a+\ba_{a,b}\bigl(Z_{a,b}^N(\alpha,\beta)\bigr).\]
By Remark~\ref{ordprod}, Proposition~\ref{bnafunc} and Example~\ref{ordA}, 
the right hand side equals $\ba_{E^{(\alpha,\beta)} \otimes \sgn_l}$. 
Hence, by Lemma~\ref{pieri2}, the assumption implies that 
\[Z_{a,b}^N(\lambda,\mu) \trianglelefteq \hat{Z}_{a,b}(\alpha,\beta)
\qquad \mbox{and}\qquad \ba_{a,b}\bigl(Z_{a,b}^N(\lambda,\mu)\bigr)=
\ba_{a,b}\bigl(\hat{Z}_{a,b}^N(\alpha,\beta)\bigr).\] 
So Remark~\ref{lem12} shows that $Z_{a,b}^N(\lambda,\mu)=\hat{Z}_{a,b}^N
(\alpha,\beta)$, as desired.
\end{proof}

Now we can state the main result of this section.

\begin{thm} \label{mainbn} Recall that $b \geq 0$ and $a>0$. Let
$(\lambda,\mu)$ and $(\lambda',\mu')$ be bipartitions of $n$. Then 
we have:
\[ E^{(\lambda,\mu)} \preceq_L E^{(\lambda',\mu')} \quad \Rightarrow \quad
(\lambda,\mu) \preceq_{a,b} (\lambda',\mu') \qquad \mbox{(see 
Definition~\ref{def1})}.\] 
\end{thm}

\begin{proof} We proceed by induction on $n$. If $n=1$, then $(1,
\varnothing)$ and $(\varnothing,1)$ are the only bipartitions of $n$
and the assertion is easily checked directly. Now let $n \geq 2$. We use 
the characterisation of $\preceq_L$ in Lemma~\ref{ordB1}. It is sufficient 
to consider an elementary step in that characterisation, that is, we can 
assume that there exists a decomposition $n=k+l$ ($k\geq 0$, $l\geq 1$) 
and $M,M'\in \Irr(W_k)$ where $M\preceq_L M'$ within $\Irr(W_k)$, such 
that one of the following conditions is satisfied:
\begin{itemize}
\item[(I)] $M \boxtimes \sgn_{l} \uparrow E^{(\lambda,\mu)}$ and
$M'\boxtimes \sgn_{l} \rightsquigarrow_L E^{(\lambda',\mu')}$.
\item[(II)] $M \boxtimes \sgn_{l}\uparrow E^{(\lambda',\mu')} 
\otimes \sgn$ and $M' \boxtimes \sgn_{l} \rightsquigarrow_L 
E^{(\lambda,\mu)}\otimes \sgn$.
\end{itemize}
Now write $M=E^{(\alpha,\beta)}$ and $M'=E^{(\alpha',\beta')}$ where 
$(\alpha,\beta)$ and $(\alpha',\beta')$ are bipartitions of $k$. 
Let $N\geq 0$ be a sufficiently large integer and consider the
multisets corresponding to the above partitions.  We also consider
the multisets $\hat{Z}_{a,b}^N(\alpha,\beta)$ and $\hat{Z}_{a,b}^N
(\alpha',\beta')$, as defined in Lemma~\ref{pieri2}. 

Since $M\preceq_L M'$, we have $(\alpha,\beta) \preceq_{a,b} (\alpha',
\beta')$ by induction. Recall that this means that $Z_{a,b}^N(\alpha,
\beta) \trianglelefteq Z_{a,b}^N(\alpha',\beta')$. By the definition of 
$\trianglelefteq$, this immediately implies that we also have 
\[ \hat{Z}_{a,b}^N(\alpha,\beta) \trianglelefteq \hat{Z}_{a,b}^N(\alpha',
\beta').\]
Now assume that (I) holds. Then, by Lemmas~\ref{pieri2} and \ref{pieri3},
we deduce that 
\[ Z_{a,b}^N(\lambda,\mu) \trianglelefteq \hat{Z}_{a,b}^N(\alpha,\beta)
\trianglelefteq \hat{Z}_{a,b}^N(\alpha',\beta')=Z_{a,b}^N(\lambda',\mu')\]
and, hence, $(\lambda,\mu) \preceq_{a,b} (\lambda',\mu')$, as required.

Now assume that (II) holds. By Lemma~\ref{bnsign}, we have
\[E^{(\lambda,\mu)} \otimes \sgn=E^{(\overline{\mu}, \overline{\lambda})}
\qquad \mbox{and}\qquad E^{(\lambda',\mu')} \otimes \sgn=
E^{(\overline{\mu}',\overline{\lambda}')}.\]
Arguing as in case (I), we obtain that $(\overline{\mu}',
\overline{\lambda}') \preceq_{a,b} (\overline{\mu}, \overline{\lambda})$.
But then Proposition~\ref{cor11} implies that we also have $(\lambda,\mu)
\preceq_{a,b} (\lambda',\mu')$, as required.
\end{proof}

\begin{rem} \label{rembn} One is tempted to conjecture that the reverse 
implication in Theorem~\ref{mainbn} also holds. In the cases where 
$(a,b)\in \{(1,1), (1,0)\}$ or $b>(n-1)a>0$, this will be shown in 
Section~\ref{sec3a} below. 

However, the reverse implication in Theorem~\ref{mainbn} does not hold in
general. The following example was found by Bonnaf\'e in connection 
with a somewhat related conjecture in \cite[Remark~1.2]{bgil}. Let $n=5$,
$b=1$ and $a=2$. Then one can check that $(\varnothing,221) \preceq_{2,1} 
(32,\varnothing)$ but $E^{(\varnothing,221)}$ and $E^{(32,\varnothing)}$ 
are not related by $\preceq_L$.

In any case, the implication in Theorem~\ref{mainbn} is sufficient
to obtain all our applications in Section~\ref{sec4}.
\end{rem}

\begin{exmp} \label{bnanull} Assume that $a=0$ and $b>0$. (See
Example~\ref{expordtriv} for the case $a=b=0$.) Then, by 
\cite[Example~1.3.9]{GeNico}, we have 
\[ \ba_{E^{(\lambda,\mu)}}=b|\mu| \quad \mbox{for all bipartitions
$(\lambda,\mu)$ of $n$}.\]
Now let $(\lambda,\mu)$ and $(\lambda',\mu')$ be two bipartitions of $n$.
Using similar methods, it is not difficult to show that 
\[ E^{(\lambda,\mu)} \preceq_L E^{(\lambda',\mu')} \qquad 
\Leftrightarrow \qquad |\mu|\geq |\mu'|;\]
furthermore, $E^{(\lambda,\mu)}$ and $E^{(\lambda',\mu')}$ belong to 
the same family if and only if $|\mu|=|\mu'|$. (As this case is not
very interesting for applications, we omit further details; see also 
\cite[Cor.~2.4.12]{GeNico}.)
\end{exmp}

%%%%%%%%%%%%%%%%%%%%%%%%%%%%%%%%%%%%%%%%%%%%%%%%%%%%%%%%%%%%%%%%%%%%%%%%%%%
\section{Examples} \label{sec3a}

We keep the notation of the previous section where $W_n$ is a Coxeter group 
of type $B_n$ and $L \colon W_n \rightarrow \Z$ is a weight function 
specified by $b:=L(t) \geq 0$ and $a:=L(s_i)>0$ for $1 \leq i \leq n-1$,
where $\{t,s_1,\ldots,s_{n-1}\}$ are the generators of $W_n$. In this 
section, we discuss some examples and further interpretations of the 
pre-order relation $\preceq_L$ on $\Irr(W_n)$, involving some interesting
combinatorics. We begin by considering the ``asymptotic case''.

\begin{exmp} \label{expasym2} Assume that $b>(n-1)a>0$, as in 
Section~\ref{secasym}. Then we claim that:
\begin{center}
\fbox{$E^{(\lambda,\mu)} \preceq_L E^{(\lambda',\mu')} \quad 
\Leftrightarrow \quad E^{(\lambda,\mu)} \leq_{\cLR} E^{(\lambda',\mu')} 
\quad \Leftrightarrow \quad (\lambda,\mu) \trianglelefteq (\lambda',\mu')$}
\end{center}
where $\trianglelefteq$ is the dominance order on bipartitions; see
Example~\ref{defdouble}. In particular, the reverse implication in 
Remark~\ref{remLR} holds. (This is a new result.)

The above equivalences are proved as follows. By Remark~\ref{remLR}, the 
first condition implies the second. As already noted in 
\cite[Example~3.7]{klord}, the second condition implies the third, as a 
consequence of \cite[Prop.~5.4]{geia06}. So it remains to prove the 
implication 
\begin{equation*}
(\lambda,\mu) \trianglelefteq (\lambda',\mu') \qquad \Rightarrow \qquad
E^{(\lambda,\mu)} \preceq_L E^{(\lambda',\mu')}.\tag{$*$}
\end{equation*}
We begin by noting that, since $b>(n-1)a$, the following simplified
version of Lemma~\ref{pieri3} holds. Let $n=k+l$ where $k \geq 0$, 
$l \geq 1$. Let $(\alpha,\beta)$ be a bipartition of $k$ and define
$\hat{\alpha}$ to be the partition of $|\alpha|+l$ obtained by increasing 
the largest $l$ parts of $\alpha$ by~$1$. Then we have 
\begin{equation*}
E^{(\alpha,\beta)} \otimes \sgn_l \rightsquigarrow_L E^{(\hat{\alpha},
\beta)}. \tag{$\dagger$}
\end{equation*}
This has already been noted in \cite[Prop.~5.2]{my02}. Indeed, by 
Lemma~\ref{pieri}, we certainly have $E^{(\alpha,\beta)} \otimes 
\sgn_l \uparrow E^{(\hat{\alpha},\beta)}$. It remains to use a 
simplified formula for the $\ba$-function, which can now be written 
as follows (see \cite[Example~3.6]{geia06}):
\[ \ba_{E^{(\lambda,\mu)}}=\bigl(n(\lambda)+2n(\mu)-n(\overline{\mu})
\bigr) \,a+|\mu|\,b.\]
After these preparations, we can now turn to the proof of ($*$). It 
will certainly be sufficient to assume that $(\lambda,\mu)$ is adjacent
to $(\lambda',\mu')$ in the dominance order. Also note that, if 
$(\lambda,\mu) \trianglelefteq (\lambda',\mu')$, then $|\lambda|\leq 
|\lambda'|$. According to Remark~\ref{remadjdouble}, this leads us 
to distinguish three cases.

{\em Case 1}. We have $\mu=\mu'$ and $\lambda$ is obtained from 
$\lambda'$ by decreasing one part by $1$ and increasing one part 
by $1$. More precisely, there are indices $1 \leq l<j\leq N$ such that,
if we write $\lambda'=(\lambda_1'\geq \lambda_2'\geq \ldots \geq\lambda_N')$,
then 
\[\lambda=(\lambda_1' \geq \ldots \lambda_{l-1}' \geq \lambda_l'-1 \geq 
\lambda_{l+1}' \geq \ldots \geq \lambda_{j-1}' \geq \lambda_j'+1 \geq 
\lambda_{j+1}'\geq \ldots \geq \lambda_N').\]
Let $\nu$ be the partition obtained by decreasing the first $l$ parts 
of $\lambda'$ by $1$. Then notice that $\lambda$ can be obtained by 
increasing $l$ parts of $\nu$ by~$1$. Now consider the representation 
$E^{(\nu,\mu)}$ of $W_{k}$ where $k=n-l$. By Lemma~\ref{pieri} and 
($\dagger$), we have
\[ E^{(\nu,\mu)} \otimes \sgn_{l} \uparrow E^{(\lambda,\mu)} \qquad
\mbox{and} \qquad E^{(\nu,\mu)} \otimes \sgn_l \rightsquigarrow_L 
E^{(\lambda',\mu)}.\]
This means that $E^{(\lambda,\mu)} \preceq_L E^{(\lambda',\mu)}$, as 
required. 

{\em Case 2}. We have $\lambda=\lambda'$ and $\mu$ is obtained from 
$\mu'$ by increasing one part by $1$ and decreasing one part by $1$. 
Then the bipartitions $(\overline{\mu}',\overline{\lambda}')$ and
$(\overline{\mu},\overline{\lambda})$ are related as in Case~1. So 
we can conclude that 
\[E^{(\lambda',\mu')} \otimes \sgn=E^{(\overline{\mu}',
\overline{\lambda}')}\preceq_L E^{(\overline{\mu},\overline{\lambda})}=
E^{(\lambda,\mu)} \otimes \sgn,\]
where we also used Lemma~\ref{bnsign}. By the definition of $\preceq_L$, 
it is then clear that $E^{(\lambda,\mu)} \preceq_L E^{(\lambda',\mu')}$ 
as required.

{\em Case 3}. We have $|\lambda|<|\lambda'|$ and $(\lambda,\mu)$ is 
obtained from $(\lambda',\mu')$ by decreasing one part of $\lambda'$ 
by~$1$ and increasing one part of $\mu'$ by~$1$. The same statement then 
also holds, of course, for $(\overline{\mu},\overline{\lambda})$ and 
$(\overline{\mu}',\overline{\lambda}')$. In particular, $\overline{\mu}$ 
is not the empty partition. Let $l$ be the number of (non-zero) parts of 
$\overline{\mu}$ and let $\nu$ be the partition obtained by decreasing 
all (non-zero) parts of $\overline{\mu}$ by $1$. Then notice that 
$\overline{\mu}'$ can be obtained by increasing some parts of $\nu$ by~$1$. 
Now consider the representation $E^{(\nu,\overline{\lambda})}$ of $W_k$ 
where $k=n-l$. By Lemma~\ref{pieri} and ($\dagger$), we have
\[ E^{(\nu,\overline{\lambda})} \otimes \sgn_l \uparrow
E^{(\overline{\mu}',\overline{\lambda}')} \qquad \mbox{and} 
\qquad E^{(\nu,\overline{\lambda})} \otimes \sgn_l \rightsquigarrow_L 
E^{(\overline{\mu},\overline{\lambda})}.\]
Using Lemma~\ref{bnsign}, we see that $E^{(\lambda,\mu)} \preceq_L 
E^{(\lambda',\mu')}$ as required.
\end{exmp}

\begin{exmp} \label{expeq} Assume that $a=1$ and $b \in 
\{0,1\}$, as in Example~\ref{expab1}. We claim that then the reverse 
implication in Theorem~\ref{mainbn} also holds, that is, we have: 
\begin{center}
\fbox{$E^{(\lambda,\mu)} \preceq_L E^{(\lambda',\mu')} \quad 
\Leftrightarrow \quad (\lambda,\mu) \preceq_{1,b} (\lambda',\mu')$.}
\end{center}
Indeed, by the discussion in Example~\ref{expab1}, it will be sufficient 
to prove the reverse implication in Theorem~\ref{mainbn} assuming that 
$(\lambda,\mu)$ and $(\lambda',\mu')$ are ``$(1,b)$-special''. But this 
has already been done by Spaltenstein \cite[\S 4]{Spalt}, using an 
explicit construction which is similar to, but much more ingenious than 
the one in Example~\ref{expasym2}. 
\end{exmp}

\begin{exmp} \label{exbgil} Let $a=2$ and $b\geq 1$ be odd.
Then $b=2r+b'$ where $r \geq 0$ and $b'=1$. This choice of parameters 
naturally arises from the representation theory of the finite unitary 
groups $G=\mbox{GU}_m(\F_q)$ where $m=2n+\frac{1}{2}r(r+1)$; see 
\cite[Remark~1.1]{bgil} and the references there. 

Let $(\lambda,\mu)$ be a bipartition of $n$. The corresponding 
multiset $Z_{2,b} (\lambda, \mu)$ is formed by the $2N+r$ entries
\begin{alignat*}{2}
2&(\lambda_{i}+N+r-i)+1 &&\qquad (1\leq i \leq N+r),\\
2&(\mu_{i}+N-i) &&\qquad (1\leq i \leq N).
\end{alignat*}
Since these entries are all distinct, we can regard this multiset as the 
set of $\beta$-numbers of a partition, which we denote by $\pi_b(\lambda,
\mu)$. Setting $a=2$ and $b'=1$ in the right-hand side of the formula 
in (M1), we obtain
\[ na+N^2a+N(b-a)+\binom{r}{2}a+rb'=2n+\frac{1}{2}r(r+1)+
\binom{2N+r}{2}.\]
This means that $\pi_b(\lambda,\mu)$ is a partition of $2n+\frac{1}{2}r(r+
1)$. Further note that two partitions are related by the dominance order 
if and only if the corresponding sets of $\beta$-numbers (arranged in 
decreasing order) are related by the dominance order. Hence, in this case, 
we can restate the implication in Theorem~\ref{mainbn} as follows:
\begin{center}
\fbox{$E^{(\lambda,\mu)} \preceq_L E^{(\lambda',\mu')} \qquad \Rightarrow 
\qquad \pi_b(\lambda,\mu) \trianglelefteq \pi_b(\lambda',\mu')$.}
\end{center}
Although the setting is somewhat different, a statement of this kind has
been conjectured by Bonnaf\'e et al. \cite[Remark~1.2]{bgil}. 
\end{exmp}

\begin{exmp} \label{expspringer} Assume that $a=b=1$,
that is, we are in the ``equal parameter case''. Then $W_n$ is the Weyl 
group of the algebraic group $G=\mbox{SO}_{2n+1}(F)$, where we can take 
the field $F$ to be $\C$ or $\overline{\F}_p$ for a prime $p>2$. In this 
setting, by the main results of \cite{klord}, the pre-order relation 
$\preceq_L$ admits a geometric interpretation in terms of the Zariski 
closure relation among the unipotent classes of $G$. This can also be used
to obtain a combinatorial description of $\preceq_L$ which, however, looks 
different from that in Theorem~\ref{mainbn}~! Let us briefly explain what 
is happening here. The Springer correspondence yields a map 
\[ \Irr(W_n) \rightarrow \{\mbox{set of unipotent classes of $G$}\},
\qquad E \mapsto O_E;\]
see \cite[\S 5]{klord} and the references there. Now assume that $E,E'\in 
\Irr(W_n)$ are ``special'', that is, they are labelled by $(1,1)$-special
bipartitions in the sense of Example~\ref{expab1}. 
Then, by \cite[Theorem~4.10 and Corollary~5.5]{klord}, we have:
\[E \preceq_L E' \qquad \Leftrightarrow\qquad E \leq_{\cLR} E' 
\qquad \Leftrightarrow \qquad O_E\subseteq\overline{O}_{E'},\]
where the bar denotes Zariski closure. Now, the Springer correspondence 
is described explicitly as follows. Consider the map 
\[\{\mbox{bipartitions of $n$}\}\rightarrow\{\mbox{partitions of 
$2n+1$}\}, \qquad (\lambda,\mu) \mapsto \pi_3(\lambda,\mu),\]
defined in terms of multisets $Z_{2,3}^N(\lambda,\mu)$ as in 
Example~\ref{exbgil}. (Note that these are not the multisets associated
with $a=b=1$.) Since $G \subseteq \mbox{GL}_{2n+1}(F)$, every unipotent 
class of $G$ consists of matrices of a fixed Jordan type which is 
specified by a partition of $2n+1$. It is known that this partition 
determines the unipotent class in $G$; see \cite[\S 13.1]{Carter2}. Then, 
as explained in \cite[\S 13.3]{Carter2}, we have
\[ O_{E^{(\lambda,\mu)}}=\mbox{unipotent class consisting of matrices
of Jordan type $\pi_3(\lambda,\mu)$}.\]
Furthermore, it is known that the closure relation among the unipotent 
classes of $G$ is given by the dominance order among the partitions 
labelling the unipotent classes; see \cite[\S 13.4]{Carter2}. Hence, we 
conclude:
\begin{center}
\fbox{$E^{(\lambda,\mu)} \preceq_L E^{(\lambda',\mu')} \quad \Leftrightarrow
\quad \pi_3(\lambda,\mu)\trianglelefteq \pi_3(\lambda',\mu')$}
\end{center}
where $(\lambda,\mu)$, $(\lambda',\mu')$ are $(1,1)$-special. Thus, via the 
Springer correspondence, we have obtained a new combinatorial description 
of $\preceq_L$. Consequently, by Example~\ref{expeq}, the following 
equivalence must be true for $(1,1)$-special bipartitions $(\lambda,\mu)$ 
and $(\lambda',\mu')$: 
\[ \pi_3(\lambda,\mu) \trianglelefteq \pi_3(\lambda',\mu') \qquad 
\Leftrightarrow \qquad (\lambda,\mu) \preceq_{1,1} (\lambda',\mu').\]
This can, of course, also be checked directly; see Example~\ref{direct}
below. 

Finally, $W_n$ can also be regarded as the Weyl group of $G=\mbox{Sp}_{2n}
(F)$. In this case, the Springer correspondence is described using the map 
\[\{\mbox{bipartitions of $n$}\}\rightarrow\{\mbox{partitions of 
$2n$}\}, \qquad (\lambda,\mu) \mapsto \pi_1(\lambda,\mu),\]
defined in terms of multisets $Z_{2,1}^N(\lambda,\mu)$. Then, arguing 
as above, one finds that 
\begin{center}
\fbox{$E^{(\lambda,\mu)} \preceq_L E^{(\lambda',\mu')}\quad\Leftrightarrow
\quad \pi_1(\lambda,\mu)\trianglelefteq \pi_1(\lambda',\mu')$}
\end{center}
where $(\lambda,\mu)$, $(\lambda',\mu')$ are $(1,1)$-special. Again, for 
such bipartitions, it must be true that $\pi_1(\lambda,\mu)\trianglelefteq 
\pi_1 (\lambda',\mu') \Leftrightarrow (\lambda,\mu) \preceq_{1,1}
(\lambda',\mu')$.
\end{exmp}

\begin{exmp} \label{direct} Here we give a direct combinatorial proof for 
the following equivalence that we encountered in Example~\ref{expspringer}
above: 
\[Z_{1,1}^N(\lambda,\mu) \trianglelefteq Z_{1,1}^N(\lambda',\mu')\quad 
\Leftrightarrow \quad Z_{2,3}^N(\lambda,\mu) \trianglelefteq 
Z_{2,3}^N(\lambda',\mu')\]
where $(\lambda,\mu)$ and $(\lambda',\mu')$ are $(1,1)$-special bipartitions 
of $n$. This will also be a good illustration of how to deal with the
multisets $Z_{a,b}^N(\lambda,\mu)$. First, we need some preparations. 
Let $(\lambda,\mu)$ be $(1,1)$-special and write
\[\lambda=(\lambda_1\geq \lambda_2\geq \cdots \geq \lambda_{N+1} \geq 0),
\qquad  \mu=(\mu_1\geq \mu_2\geq \ldots \geq \mu_N \geq 0)\]
for some $N \geq 0$. Then $Z_{1,1}^N(\lambda,\mu)=\{z_1,z_2, \ldots, 
z_{2N+1}\}$ where 
\begin{alignat*}{2}
z_{2i-1}&=\lambda_{i}+N+1-i &&\qquad (1\leq i \leq N+1),\\
z_{2i}&=\mu_{i}+N-i &&\qquad (1\leq i \leq N).
\end{alignat*}
Since $(\lambda,\mu)$ is $(1,1)$-special, we have $z_1\geq z_2 \geq \ldots 
\geq z_{2N+1}\geq 0$. Then note that 
\[ Z_{2,3}^N(\lambda,\mu)=\{\tilde{z}_1, \tilde{z}_2, \ldots,
\tilde{z}_{2N+1}\}\]
where the sequence $(\tilde{z}_i)_{1\leq i \leq 2N+1}$ is defined as
follows:
\[ \tilde{z}_i=\left\{\begin{array}{cl} 
2z_i+1 & \qquad \mbox{if $i$ is odd and $z_{i-1}>z_i$},\\
2z_i & \qquad \mbox{if $i$ is even and $z_i>z_{i+1}$},\\
2z_i+1 & \qquad \mbox{if $i$ is even and $z_i=z_{i+1}$},\\ 
2z_i & \qquad \mbox{if $i$ is odd and $z_i=z_{i-1}$}.\end{array} \right.\]
Since $z_i\geq z_{i+1}$ for all $i$, one immediately checks that 
$\tilde{z}_i\geq \tilde{z}_{i+1}$ for all $i$. For $d \in \{1,\ldots,
2N+1\}$, define 
\[ \varepsilon_d=\left\{\begin{array}{cl} 1 & \quad \mbox{if $d$ is 
even and $z_d=z_{d+1}$},\\ 0 & \quad \mbox{otherwise}. \end{array}\right.\]
Then we find that 
\begin{equation*}
\sum_{1\leq i \leq d} \tilde{z}_i=\varepsilon_d+ \lfloor (d+1)/2\rfloor+
2\Bigl(\sum_{1\leq i \leq d} z_i\Bigr),\tag{$*$}
\end{equation*}
where $\lfloor x \rfloor$ denotes the integer part of $x$.
Now let $(\lambda',\mu')$ also be $(1,1)$-special and define 
\[ Z_{1,1}^N(\lambda', \mu')=\{z_1',z_2',\ldots,z_{2N+1}'\}\quad\mbox{and}
\quad Z_{2,3}^N(\lambda',\mu')= \{\tilde{z}_1',\tilde{z}_2',\ldots, 
\tilde{z}_{2N+1}'\}\]
analogously. If $Z_{2,3}^N(\lambda,\mu) \trianglelefteq Z_{2,3}^N(\lambda',
\mu')$, then
\[ \sum_{1\leq i \leq d} \tilde{z}_i \leq \sum_{1\leq i \leq d} 
\tilde{z}_i' \qquad \mbox{for $1\leq d \leq 2N+1$}.\]
Using ($*$), we deduce that $\varepsilon_d+2\sum_{1\leq i \leq d} z_i
\leq \varepsilon_d'+ 2\sum_{1\leq i \leq d} z_i'$ for all $d$. This 
certainly implies that 
\[ \sum_{1\leq i\leq d} z_i \leq \sum_{1\leq i \leq d} z_i' \qquad
\mbox{(for all $d$)} \qquad \mbox{and so} \qquad
Z_{1,1}^N(\lambda,\mu)\trianglelefteq Z_{1,1}^N(\lambda',\mu').\]
Conversely, assume that $Z_{1,1}^N(\lambda) \trianglelefteq
Z_{1,1}^N(\lambda',\mu')$. Let $d \geq 1$. Then ($*$) shows:
\[ \varepsilon_d \leq \varepsilon_d' \quad \mbox{or} \quad 
\sum_{1\leq i \leq d} z_i <\sum_{1\leq i \leq d} z_i' \qquad
\Rightarrow \qquad \sum_{1\leq i \leq d} \tilde{z}_i \leq 
\sum_{1\leq i \leq d} \tilde{z}_i'.\]
So the only case which requires an extra argument is when 
\[ \varepsilon_d=1, \qquad \varepsilon_d'=0 \qquad \mbox{and}
\qquad \sum_{1\leq i \leq d} z_i=\sum_{1\leq i \leq d} z_i'.\]
Since $\varepsilon_d=1$, we have $z_d=z_{d+1}$ and $d$ is even. We have 
\[ \sum_{1\leq i \leq d+1} z_i \leq \sum_{1\leq i \leq d+1} z_i' 
\qquad \mbox{and so} \qquad z_d=z_{d+1} \leq z_{d+1}'<z_d',\]
where the latter inequality holds since $\varepsilon_d' =0$. On the other 
hand, we also have
\[ \sum_{1\leq i \leq d-1} z_i \leq \sum_{1\leq i \leq d-1} z_i'
\qquad \mbox{and so} \qquad z_d\geq z_d',\]
a contradiction. Hence, the special case does not occur. This completes
the proof of the desired equivalence.
\end{exmp}

\begin{exmp} \label{expdn} Assume that $a=1$ and $b=0$. We
shall now explain how this case is related to groups of type $D_n$. Let 
$\tilde{W}_n\subseteq W_n$ be the subgroup generated by the reflections
$\{ts_1t,s_1, \ldots,s_{n-1}\}$ and let $\tilde{L}$ denote the restriction 
of $L$ to $\tilde{W}_n$. Then $\tilde{W}_n$ is of type $D_n$ and 
$\tilde{L}$ is the usual length function on $\tilde{W}_n$; see, for 
example, \cite[\S 1.4]{gepf}. The irreducible representations of
$\tilde{W}_n$ are classified as follows. Given a bipartition $(\lambda,
\mu)$ of $n$, we denote by $E^{[\lambda,\mu]}$ the restriction of 
$E^{(\lambda,\mu)}\in\Irr(W_n)$ to $\tilde{W}_n$. Then we have (see 
\cite[5.6.1]{gepf}):
\begin{itemize}
\item If $\lambda \neq \mu$, then $E^{[\lambda,\mu]}=E^{[\mu,\lambda]}$
is an irreducible representation of $\Irr(\tilde{W}_n)$.
\item If $\lambda=\mu$, then $E^{[\lambda,\lambda]}=E^{[\lambda,+]}\oplus
E^{[\lambda,-]}$ where $E^{[\lambda,+]}$ and $E^{[\lambda,-]}$ are 
non-isomorphic irreducible representations of $\tilde{W}_n$. (This can
only occur if $n$ is even.)
\end{itemize}
Furthermore, all irreducible representations of $\tilde{W}_n$ arise in
this way. We shall say that $\tilde{E}\in \Irr(\tilde{W}_n)$ is ``special''
if $\tilde{E}$ is a constituent of $E^{[\lambda,\mu]}$ where $(\lambda,
\mu)$ is $(1,0)$-special in the sense of Example~\ref{expab1}. This 
coincides with Lusztig's definition \cite[Chap.~4]{LuBook}. In particular, 
each family of $\Irr(\tilde{W}_n)$ contains a unique special representation
and so it is enough to describe $\preceq_{\tilde{L}}$ for special 
representations. 

Now $\tilde{W}_n$ is the Weyl group of the algebraic group $\tilde{G}=
\mbox{SO}_{2n}(F)$ where, as above, $F$ is $\C$ or $\overline{\F}_p$ for 
a prime $p>2$. Again, the Springer correspondence yields a map 
\[ \Irr(\tilde{W}_n) \rightarrow \{\mbox{set of unipotent classes of 
$\tilde{G}$}\}, \qquad \tilde{E} \mapsto \tilde{O}_{\tilde{E}},\]
which is explicitly described in \cite[\S 13.3]{Carter2}. By the main 
results of \cite{klord}, we have 
\[\tilde{E} \preceq_{\tilde{L}} \tilde{E}' \qquad \Leftrightarrow\qquad 
\tilde{E} \leq_{\cLR} \tilde{E}' \qquad \Leftrightarrow \qquad 
\tilde{O}_{\tilde{E}}\subseteq\overline{\tilde{O}}_{\tilde{E}'}\]
for any $\tilde{E},\tilde{E}' \in \Irr(\tilde{W}_n)$ which are special. 
By Spaltenstein \cite[\S 4]{Spalt}, the condition on the right hand side 
can be expressed using $\preceq_{1,0}$. More precisely, let 
$(\lambda,\mu)$ and $(\lambda',\mu')$ be $(1,0)$-special bipartitions
of $n$ such that $\tilde{E}$ is a constituent of $E^{[\lambda,\mu]}$ 
and $\tilde{E}'$ is a constituent of $E^{[\lambda',\mu']}$. Here, 
$(\lambda,\mu)$ and $(\lambda',\mu')$ are uniquely determined. (Just note
that, if both $(\lambda,\mu)$ and $(\mu,\lambda)$ are $(1,0)$-special,
then $\lambda=\mu$.) Then 
\[ \tilde{E} \preceq_{\tilde{L}} \tilde{E}' \qquad \Leftrightarrow \qquad
\left\{\begin{array}{cl} \tilde{E}=\tilde{E}' & \qquad \mbox{if
$\lambda=\lambda'=\mu=\mu'$},\\ (\lambda,\mu) \preceq_{1,0} 
(\lambda',\mu') & \qquad \mbox{otherwise}.\end{array}\right.\]
Thus, $\preceq_{\tilde{L}}$ is essentially determined by $\preceq_{1,0}$, 
where some special care is required when comparing the two representations
$E^{[\lambda,\pm]}$ with each other. 
\end{exmp}

%%%%%%%%%%%%%%%%%%%%%%%%%%%%%%%%%%%%%%%%%%%%%%%%%%%%%%%%%%%%%%%%%%%%%%%%%%%
\section{Concluding remarks} \label{sec4}
We return to the general setting where $W$ is any finite Coxeter group
and $L \colon W \rightarrow \Z$ is a weight function such that $L(s)
\geq 0$ for $s \in S$. Having established the results in 
Section~\ref{sec3}, we can now formulate some properties of the 
pre-order relation $\preceq_L$ which hold in complete generality.

\begin{thm} \label{ordmon} Let $E,E'\in \Irr(W)$. If $E \preceq_L E'$, 
then $\ba_{E'} \leq \ba_E$, with equality only if $E,E'$ belong to 
the same family.
\end{thm}

\begin{proof} Using Remark~\ref{ordprod}, it is sufficient to prove this
in the case where $(W,S)$ is irreducible. So let us assume that $(W,S)$
is irreducible. If $L(s)=0$ for all $s \in S$, then the assertions 
trivially hold by Example~\ref{expordtriv}. Assume now that $L(s)>0$ 
for some $s \in S$. 

If $W$ is of type $A_n$, $D_n$, $H_3$, $H_4$, $E_6$, $E_7$, $E_8$ 
or $I_2(m)$ ($m$ odd), then we are in the equal parameter case and 
the assertions hold by \cite[Example~3.5 and Prop.~4.4]{klord}. 

If $W$ is of type $I_2(m)$ ($m$ even) or $F_4$, the assertions can 
be checked by explicit computations; see \cite{compf4}
and \cite[Examples~3.5 and 3.6]{klord}. If $L(s)=0$ for some $s\in S$,
see \cite[Example~2.2.8 and Remark~2.4.13]{GeNico} where similar 
verifications have been performed with respect to the pre-order 
relation $\leq_{\cLR}$. 

Finally, assume that $W$ is of type $B_n$ with parameters $a,b$ as
in Section~\ref{sec3}. If $a>0$, the assertions follow from
Theorem~\ref{mainbn} and Remark~\ref{lem12}, in combination 
with Propositions~\ref{bnafunc} and \ref{bnfam}. For the case $a=0$,
see Example~\ref{bnanull}. 
\end{proof}

\begin{cor} \label{ordfam} Let $E,E'\in \Irr(W)$. Then $E$ and $E'$ belong 
to the same family if and only if $E \preceq_L E'$ and $E' \preceq_L E$.
\end{cor}

\begin{proof} The ``only if'' part is clear by the definitions, as
already mentioned in \cite[Remark~2.11]{klord}. The ``if'' part now
immediately follows from Theorem~\ref{ordmon}. 
\end{proof}

\begin{cor} \label{ordomg} Let $E,E'\in \Irr(W)$. If $E \preceq_L E'$, 
then $\omega_L(E) \leq \omega_L(E')$, with equality only if $E,E'$ belong 
to the same family.
\end{cor}

\begin{proof} Assume that $E \preceq_L E'$. By the definition of $\preceq_L$, 
it is clear that then we also have $E' \otimes \sgn \preceq_L E\otimes\sgn$.
Hence, using Remark~\ref{defafun1} and Theorem~\ref{ordmon}, we deduce that
$\omega_L(E) \leq \omega_L(E')$, with equality only if $E,E'$ belong to the
same family.
\end{proof}

\begin{rem} \label{remAa} Assume that $W$ is a Weyl group and we are in the 
equal parameter case. For $E \in \Irr(W)$, let $D_E(u) \in {\Q}[u]$ (where 
$u$ is an indeterminate) be the corresponding generic degree, defined in 
terms of the associated Iwahori--Hecke algebra; see, for example, 
\cite[Cor.~9.3.6]{gepf}. Then the invariant $\ba_E$ is given by
\[D_E(u)=f_E^{-1}u^{\ba_E}+\mbox{ linear combination of higher powers 
of $u$},\]
where $f_E$ is a non-zero integer; see Lusztig \cite[4.1]{LuBook}.
Furthermore, we have
\[D_E(u)=f_E^{-1}u^{\bA_E}+\mbox{ linear combination of lower powers 
of $u$},\]
for some integer $\bA_E \geq 0$. By \cite[Prop.~9.4.3]{gepf}, we have 
$\omega_L(E)=L(w_0)-(\ba_E+\bA_E)$.
Hence, Corollary~\ref{ordomg} implies that
\[ E \preceq_L E'\quad\Rightarrow\quad\ba_{E'}+\bA_{E'}\geq \ba_E+\bA_E,\]
with equality only if $E,E'$ belong to the same family. An analogous
statement, where ``families'' and ``$\preceq_L$'' are replaced by 
``two-sided cells'' and the relation ``$\leq_{\cLR}$'' (as referred 
to in Remark~\ref{remLR}), plays a role in Ginzburg et al., 
\cite[Prop.~6.7]{ggor}. 
\end{rem}

%USE induction on the absolute value of $\omega_L(E)$ to show that, in
%type $B_n$ with $b'>0$, for every $E \in \Irr(W)$ there exists a left
%cell $C$ such that $E=[C]_1$.????

%%%%%%%%%%%%%%%%%%%%%%%%%%%%%%%%%%%%%%%%%%%%%%%%%%%%%%%%%%%%%%%%%%%%%%%%%%%
 
\end{document}